\documentclass[11pt,reqno]{amsart}
\usepackage{amscd,amsmath,amsopn,amssymb,amsthm,multicol}
\usepackage[breaklinks=true,colorlinks=true,linkcolor=blue,citecolor=blue,urlcolor=blue]{hyperref}

\DeclareMathOperator{\ad}{ad}

\DeclareMathOperator{\diag}{diag}

\DeclareMathOperator{\Ad}{Ad}

\DeclareMathOperator{\Lin}{Lin}
\DeclareMathOperator{\Aut}{Aut}

\DeclareMathOperator{\trace}{trace}

\DeclareMathOperator{\End}{End}

\renewenvironment{proof}[1][Proof]{\textbf{#1.} }
{\ \rule{0.5em}{0.5em}}
\newtheorem{theorem}{Theorem}
\newtheorem{prop}{Proposition}
\newtheorem{lemma}{Lemma}
\newtheorem{corollary}{Corollary}
\newtheorem{quest}{Question}

\theoremstyle{definition}
\newtheorem{definition}{Definition}
\newtheorem{remark}{Remark}
\newtheorem{example}{Example}

\begin{document}

\title
[On the structure of geodesic orbit \dots] {On the structure of geodesic orbit Riemannian  spaces}

\author{Yu.G.~Nikonorov}

\begin{abstract}

The paper is devoted to the study of geodesic orbit Riemannian spaces that could be characterize by the property that
any geodesic is an orbit of a 1-parameter group of isometries. In particular, we discuss some important totally geodesic submanifolds that inherit
the property to be geodesic orbit. For a given geodesic orbit Riemannian space, we describe the structure of  the nilradical and the radical of the Lie
algebra of the isometry group. In the final part, we discuss some
new tools to study geodesic orbit Riemannian spaces, related to compact Lie group representations with non-trivial principal isotropy algebras.
We discuss also some new examples of geodesic orbit Riemannian spaces, new methods to obtain such examples, and some unsolved questions.

\vspace{2mm} \noindent 2010 Mathematical Subject Classification:
53C20, 53C25, 53C35.

\vspace{2mm} \noindent Key words and phrases: homogeneous Riemannian manifolds, symmetric spaces, homogeneous spaces, geodesic orbit Riemannian spaces.
\end{abstract}

\maketitle

\section{Introduction, notation and useful facts}

All manifolds in this paper are supposed to be connected.
At first, we recall and discuss important definitions.

\begin{definition}\label{GOman}
A Riemannian manifold $(M,g)$ is called a  manifold with
homogeneous geodesics or geodesic orbit manifold (shortly,  {\it GO-manifold}) if any
geodesic $\gamma $ of $M$ is an orbit of a 1-parameter subgroup of
the full isometry group of $(M,g)$.
\end{definition}

\begin{definition}\label{GOsp}
A Riemannian manifold $(M=G/H,g)$, where $H$ is a compact subgroup
of a Lie group $G$ and $g$ is a $G$-invariant Riemannian metric,
is called a space with homogeneous geodesics or geodesic orbit space
(shortly,  {\it GO-space}) if any geodesic $\gamma $ of $M$ is an orbit of
a 1-parameter subgroup of the group $G$.
\end{definition}

Hence, a Riemannian manifold $(M,g)$ is  a geodesic orbit Riemannian manifold,
if it is a geodesic orbit space with respect to its full connected isometry group.
This terminology was introduced in
\cite{KV} by O.~Kowalski and L.~Vanhecke, who initiated a systematic study of such spaces.

Let $(M=G/H, g)$ be a homogeneous Riemannian manifold. Since $H$
is compact, there is an $\Ad(H)$-invariant decomposition
\begin{equation}\label{reductivedecomposition}
\mathfrak{g}=\mathfrak{h}\oplus \mathfrak{m},
\end{equation}
where $\mathfrak{g}={\rm Lie }(G)$ and $\mathfrak{h}={\rm Lie}(H)$.
The Riemannian metric $g$ is $G$-invariant and is determined
by an $\Ad(H)$-invariant inner product $g = (\cdot,\cdot)$ on
the space $\mathfrak{m}$ which is identified with the tangent
space $T_oM$ at the initial point $o = eH$.

By $[\cdot, \cdot]$ we denote the Lie bracket in $\mathfrak{g}$, and by
$[\cdot, \cdot]_{\mathfrak{m}}$ its $\mathfrak{m}$-component according to (\ref{reductivedecomposition})
We recall (in the above terms) a well-known criteria of GO-space.

\begin{lemma}[\cite{KV}]\label{GO-criterion}
A homogeneous Riemannian manifold   $(M=G/H,g)$ with the reductive
decomposition  {\rm (\ref{reductivedecomposition})} is a geodesic orbit space if and
only if  for any $X \in \mathfrak{m}$ there is $Z \in \mathfrak{h}$ such that
$([X+Z,Y]_{\mathfrak{m}},X) =0$ for all $Y\in \mathfrak{m}$.
\end{lemma}

In what follows, the latter condition in this lemma will be called the {\it GO-property}. Recall that for a given $X\in\mathfrak{m}$, this means that the orbit
of  $\exp\bigl((X+Z)t\bigr) \subset G$, $t \in \mathbb{R}$, through the point $o=eH$ is a geodesic in $(M=G/H,g)$.
Note also that all orbits of a 1-parameter isometry group, generated by a Killing vector fields of constant length on a given Riemannian manifold,
are geodesics, see  \cite{BerNikKF, Nik2015}
and the references therein.
\medskip

There are some important subclasses of geodesic orbit manifolds.
Indeed, GO-spaces  may be considered as  a natural generalization of Riemannian symmetric spaces,
introduced and classified by \`{E}.~Cartan in \cite{Ca}.
On the other hand, the class of GO-spaces is much larger than the class of symmetric spaces.
Any homogeneous space
$M = G/H$ of a compact Lie group $G$  admits a Riemannian metric $g$ such that $(M,g)$ is a GO-space.
It suffices to take the metric
$g$ induced by a biinvariant Riemannian metric $g_0$ on the Lie group  $G$ such that
$ (G,g_0) \to (M=G/H, g)$ is a Riemannian submersion
with totally geodesic fibres. Such geodesic orbit space $(M = G/H, g)$ is called  a {\it normal homogeneous space}
(in the sense of M.~Berger \cite{Berg}).

It should be noted also that  any  naturally  reductive  Riemannian manifold is geodesic orbit.
Recall that a Riemannian manifold $(M,g)$ is
{\it naturally reductive} if it admits a transitive Lie group $G$ of isometries with a biinvariant
pseudo-Riemannian metric $g_0$, which induces the metric $g$ on $M = G/H$ (see  \cite{Bes} and \cite{KN}).
Clear that symmetric spaces and normal homogeneous spaces are naturally reductive.
The classification of  naturally reductive homogeneous spaces of $\dim \leq 5$ was
obtained by O. Kowalski and L.~Vanhecke in 1985 \cite{KV1}.
New approach was suggested by
I. Agricola, A.C. Ferreira, and T. Friedrich  in the paper \cite{AFF}.
This approach allowed to get classification of naturally reductive homogeneous spaces of $\dim \leq 6$ \cite{AFF}.
Recent interesting results on naturally reductive homogeneous spaces one can find on \cite{AF, Storm} and in the references therein.

The first example of non naturally reductive GO-manifold had been constructed by  A.~Kaplan \cite{Kap}.
In \cite{KV}, O.~Kowalski and L.~Vanhecke classified all geodesic orbit spaces
of dimension $\leq 6$. In particular, they proved that every GO-manifold of dimension $\leq 5$ is naturally reductive.

An important class of GO-spaces  consists  of weakly symmetric spaces,  introduced by A.~Selberg \cite{S}.
A homogeneous Riemannian manifold  $(M = G/H, g)$ is a {\it  weakly symmetric space}
if any two points $p,q \in M$ can be interchanged by
an isometry $a \in G$. This  property does not depend on the particular $G$-invariant metric $g$.
Weakly symmetric spaces $M= G/H$
have many interesting properties
and  are closely related  with spherical spaces, commutative
spaces, Gelfand pairs etc. (see papers \cite{AV, Yakimova} and book \cite{W1} by J.A.~Wolf).
The classification of weakly symmetric
reductive homogeneous  Riemannian spaces was  given by O.S.~Yakimova \cite{Yakimova} on the base of the paper \cite{AV}
(see also  \cite{W1}). Note that
{\it weakly symmetric Riemannian manifolds are geodesic orbit}
by a result of J.~Berndt, O.~Kowalski, and L.~Vanhecke \cite{BKV}.

{\it Generalized normal homogeneous Riemannian manifolds}
({\it $\delta$-homogeneous manifold}, in another terminology) constitute another important
subclass of geodesic orbit manifolds.
All metrics from this subclass are of non-negative sectional curvature and have some other interesting properties
(see details in \cite{BerNik, BerNik3, BerNik2012}).
In the paper \cite{BerNik2012}, a classification of  generalized normal
homogeneous metrics on spheres and projective spaces is obtained.
Finally, we notice  that {\it Clifford--Wolf homogeneous
Riemannian manifolds} constitute a partial subclass of generalized normal homogeneous Riemannian manifolds \cite{BerNikClif}.

C.~Gordon \cite{Gor96} developed the theory of geodesic orbit Riemannian nilmanifolds.
Remarkably, only commutative or two-step nilpotent Lie group admits geodesic orbit Riemannian metrics.
In \cite{Gor96}, one can find also some structural results on geodesic orbit spaces and some
examples of geodesic orbit Riemannian manifolds.
In the paper \cite{AA} by D.V.~Alekseevsky and A.~Arvanitoyeorgos, the classification
of  non-normal invariant geodesic orbit metrics on  simply connected flag manifolds $M =G/H$ was  given.
In the paper~\cite{AN}, this result was generalized to the  case of compact homogeneous manifolds
with positive Euler  characteristic.

Various constructions of geodesics, that are orbits of 1-parameter isometry groups, and remarkable properties of geodesic orbit metrics
could be found also in \cite{CHNT,DuKoNi,Nik2013n,Tam,Zil96}, and in the references therein.
Notice the classification of geodesic orbit spaces fibered over irreducible symmetric spaces obtained by H.~Tamaru in \cite{Ta}
and the classification of geodesic orbit metrics on spheres in \cite{Nik2013}.
One can find also some interesting results on geodesic orbit pseudo-Riemannian manifold in
\cite{Barco,ChenWolf2012,Du2010} and the references therein.
It should be noted that there are corresponding notions and results on geodesic orbit Finsler manifolds,
see the paper \cite{YanDeng} by Z.~Yan and Sh.~Deng.
\medskip

The paper is organized as follows: In Section 2 we discuss on the Lie algebra level the algebraic structure
of an arbitrary homogeneous Riemannian space, pointing out various possibilities to choose a reductive complement
to the isotropy subalgebra in the Lie algebra of the motion group. In Section 3 we discuss totally geodesic submanifolds
of geodesic orbit Riemannian spaces and their isometry groups.
Section 4 is devoted to the study of the nilradical and the radical of the Lie algebra of the motion  group of a geodesic orbit Riemannian space.
In particular, a remarkable property of the Killing form on the nilradical of a GO-space is proved in Theorem \ref{strucrad1}, and Theorem  \ref{rsubm2}
gives a natural construction of a GO-space with the radical of very special type from a given GO-space. In Section 5 we prove
Theorems \ref{centaction1} and \ref{centaction2}, that give a new
tools to the study of GO-spaces, related to compact Lie group representations with non-trivial principal isotropy algebra
(non-trivial stationary subalgebra of points in general position).
Finally, in Conclusion we discuss some unsolved questions.
\medskip

\section{On the structure of homogeneous Riemannian space}

We recall some well-known facts on the structure of homogeneous Riemannian spaces, see e.~g. \cite{BB1978, BB1981, ChEb, GorWil1985, Wolf1969}.
The description of possible reductive decompositions for a~Riemannian homogeneous space could be helpful in various problems.

Let $(G/H, \rho)$ be a homogeneous Riemannian manifold, where $G$ is a connected Lie group, $H$ is a compact subgroup in $G$, and $\rho$ is a
$G$-invariant Riemannian metric on $G/H$. We will suppose that $G$ acts effectively on $G/H$ (otherwise it is possible to
factorize by $U$, the maximal normal subgroup of $G$ in $H$).

By the symbols $R(G)$ and $N(G)$ we denote the radical
(the largest connected solvable normal subgroup)
and the largest connected nilpotent normal subgroup
of the Lie group $G$  respectively.
For the Lie algebra $\mathfrak{g}=\operatorname{Lie}(G)$, we denote by $\mathfrak{n(g)}$ and $\mathfrak{r(g)}$ the nilradical and the radical of $\mathfrak{g}$
respectively.
A maximal semi-simple subalgebra of $\mathfrak{g}$ is called a Levi factor or a~Levi subalgebra.
There is a semidirect decomposition $\mathfrak{g}=\mathfrak{r(g)}\rtimes \mathfrak{s}$, where $\mathfrak{s}$ is an arbitrary Levi factor.
The Malcev~--~Harish-Chandra theorem
states that any two Levi factors of $\mathfrak{g}$ are conjugate by an automorphism $\exp(\Ad (Z))$
of $\mathfrak{g}$, where $Z$ is in the nilradical $\mathfrak{n(g)}$ of~$\mathfrak{g}$.

Using the structure of representations of Lie algebras (see e.~g. Proposition 5.5.17 in \cite{HilNeeb}), it is easy to see that
$\mathfrak{r(g)}=[\mathfrak{s},\mathfrak{r(g)}] \oplus C_{\mathfrak{r(g)}}(\mathfrak{s})$ (the direct sum of linear subspaces),
where $C_{\mathfrak{r(g)}}(\mathfrak{s})$ is the centralizer of
$\mathfrak{s}$ in $\mathfrak{r(g)}$.
We have $[\mathfrak{s},\mathfrak{r(g)}]\subset [\mathfrak{g},\mathfrak{r(g)}]\subset \mathfrak{n(g)}$, moreover,
$D(\mathfrak{r(g)})\subset \mathfrak{n(g)}$ for every derivation
$D$ of $\mathfrak{g}$
(Theorem 2.13 and Theorem 3.7 in~\cite{Jac}).

For an arbitrary Levi factor $\mathfrak{s}$, we have
$[\mathfrak{g},\mathfrak{g}]=[\mathfrak{r(g)}+\mathfrak{s},\mathfrak{r(g)}+\mathfrak{s}]\subset \mathfrak{n(g)}\rtimes \mathfrak{s}$.
Hence, $[\mathfrak{g},\mathfrak{g}]\cap \mathfrak{r(g)}=[\mathfrak{g},\mathfrak{r(g)}]\subset \mathfrak{n(g)}$.
Note also that any Levi factor is in $\mathfrak{n(g)}\rtimes \mathfrak{s}$ and
$[\mathfrak{g},\mathfrak{g}]=[\mathfrak{g},\mathfrak{r}]\rtimes \mathfrak{s}$ for any Levi subalgebra $\mathfrak{s}$.
If $B$ is the Killing form of $\mathfrak{g}$, then
\begin{equation}\label{nilid}
\mathfrak{n(g)} \subset \{X\in \mathfrak{g}\,|\, B(X,Y)=0\, \forall\, Y\in \mathfrak{g}\}\subset
\{X\in \mathfrak{g}\,|\, B(X,Y)=0\, \forall\, Y\in [\mathfrak{g},\mathfrak{g}]\}=\mathfrak{r(g)},
\end{equation}
see e.~g. Proposition 1.4.6 in \cite{Bourb} and Theorem 3.5  in \cite{Jac}. In particular, $B([\mathfrak{g},\mathfrak{r(g)}],\mathfrak{g})=0$.
\smallskip

Recall that a subalgebra $\mathfrak{k}$ of a Lie algebra $\mathfrak{g}$ is said to
be {\it compactly embedded in} $\mathfrak{g}$ if $\mathfrak{g}$ admits an inner product relative to which the
operators $\ad(X):\mathfrak{g} \mapsto \mathfrak{g}$, $X\in \mathfrak{k}$, are skew-symmetric.
This condition is equivalent to the following condition:
the closure of $\Ad_G(\exp(\mathfrak{k}))$ in $\operatorname{Aut}(\mathfrak{g})$ is compact,
see e.~g. \cite{HilNeeb, Wust}.
Note that for a compactly embedded subalgebra $\mathfrak{k}$, every operator $\ad(X):\mathfrak{g}\rightarrow \mathfrak{g}$,
$X\in \mathfrak{k}$, is semisimple and the spectrum  of  $\ad(X)$ lies in $i\mathbb{R}$.
A subgroup $K$ of the Lie group $G$ is {\it compactly embedded} if
the closure of $\Ad_G(K)$ in $\operatorname{Aut}(\mathfrak{g})$ is compact. It is known that a connected subgroup $K$ is compactly embedded in $G$ if and only
if $\mathfrak{k}=\operatorname{Lie}(K)$ is compactly embedded in $\mathfrak{g}=\operatorname{Lie}(G)$,
and the maximal compactly embedded subgroup of a connected Lie group is closed and connected~\cite{Wust}.
Recall also that subalgebra $\mathfrak{k}$ of a Lie algebra $\mathfrak{g}$ is said to
be {\it compact} if it is compactly embedded in itself.
It is equivalent to the fact that there is a compact Lie group with a given Lie algebra $\mathfrak{k}$.
Clear that any compactly embedded subalgebra $\mathfrak{k}$ of  $\mathfrak{g}$ is compact.

\smallskip

\begin{lemma}
The Killing form $B$ of $\mathfrak{g}$ is negative semi-definite on every compactly embedded subalgebra
{\rm(}in particular, on Lie algebras of compact subgroups in $G${\rm)}
$\mathfrak{k}$ of $\mathfrak{g}$.
If the corresponding Lie subgroup $K$ is the isotropy group in the isometry group $G$ of an effective Riemannian space $(G/K, \rho)$,
then $B$  is negative definite on $\mathfrak{k}$.
\end{lemma}

\begin{proof}
Indeed, there is an $\ad(\mathfrak{k})$-invariant inner product $Q$ on $\mathfrak{g}$ (by the definition of compactly embedded subalgebra),
hence, every operator $\ad(X):\mathfrak{g} \rightarrow \mathfrak{g}$ is skew-symmetric with respect $Q$ and $B(X,X)$
is a trace of a squared skew-symmetric matrix. Hence $B(X,X)\leq 0$ and $B(X,X)=0$ if and only if $X$ is in the center of $\mathfrak{g}$.
The latter is impossible for the isotropy algebra $\mathfrak{k}$ because $G/K$ is assumed to be effective.
\end{proof}

\bigskip

Let $G/H$ be effective homogeneous space with compact $H$. Note that $\Ad_G(H)$ (that will be also denoted as $\Ad(H)$)
is compact, hence fully reducible, group of automorphisms of the Lie algebra
$\mathfrak{g}=\operatorname{Lie}(G)$.
A fully reducible group of automorphisms of a Lie algebra
keeps a maximal semi-simple subalgebra invariant \cite{Mostov1956}.
Therefore, there is a Levi factor $\mathfrak{s}$ of  $\mathfrak{g}$ invariant with respect $\Ad_G(H)$.
Then we have  the following Levi decomposition:
\begin{equation}\label{levidec1}
\mathfrak{g}=\mathfrak{r(g)}+\mathfrak{s},
\end{equation}
where both $\mathfrak{r(g)}$ and $\mathfrak{s}$ are invariant with respect to $\Ad_G(H)$. For any $X \in \mathfrak{h}=\operatorname{Lie}(H)$,
there exists a unique decomposition

\begin{equation}\label{levidec2}
X=X_{\mathfrak{r(g)}}+X_{\mathfrak{s}},\quad \mbox{where} \quad X_{\mathfrak{r(g)}}\in \mathfrak{r(g)}, \, X_{\mathfrak{s}}\in \mathfrak{s}.
\end{equation}
Since $[X,\mathfrak{r(g)}]\subset \mathfrak{r(g)}$, $[X,\mathfrak{s}] \subset \mathfrak{s}$, and $[\mathfrak{g},\mathfrak{r(g)}]\subset \mathfrak{r(g)}$,
we get $[X_{\mathfrak{r(g)}},\mathfrak{s}]=0$ for all
$X\in \mathfrak{h}$. Hence for any $X,Y \in  \mathfrak{h}$ we get
$$
[X,Y]=[X_{\mathfrak{r(g)}}+X_{\mathfrak{s}}, Y_{\mathfrak{r(g)}}+Y_{\mathfrak{s}}]=[X_{\mathfrak{r(g)}},Y_{\mathfrak{r(g)}}]+[X_{\mathfrak{s}},Y_{\mathfrak{s}}].
$$
Therefore, we get {\it the Lie algebra homomorphisms $\varphi:\mathfrak{h} \rightarrow \mathfrak{r(g)}$ and $\psi:\mathfrak{h} \rightarrow \mathfrak{s}$}
such that
$\varphi(X)=X_{\mathfrak{r(g)}}$ and $\psi(X)=X_{\mathfrak{s}}$. Note that $\varphi(\mathfrak{h})$ is a Lie subalgebra of $\mathfrak{r(g)}$ such that
$[\varphi(\mathfrak{h}),\mathfrak{s}]=0$.

Note that $\mathfrak{h}$ is reductive, i.~e. $\mathfrak{h}=\mathfrak{c(h)}\oplus[\mathfrak{h},\mathfrak{h}]$, where
$\mathfrak{c(h)}$ is the center and $[\mathfrak{h},\mathfrak{h}]$ is a semi-simple ideal in $\mathfrak{h}$.
Since $\mathfrak{r(g)}$ is solvable Lie algebra, it can not contain any semisimple Lie subalgebra, hence $[\mathfrak{h},\mathfrak{h}] \subset
\operatorname{Ker}(\varphi) \subset \mathfrak{s}$. In particular, $\varphi(\mathfrak{h})$ is an abelian Lie algebra.

It is clear that $\mathfrak{h}+\mathfrak{s}=\mathfrak{s}\oplus \varphi(\mathfrak{h})$ is a reductive Lie algebra. Indeed, $\varphi(\mathfrak{h})$ is the center
of this algebra ($[\varphi(\mathfrak{h}),\mathfrak{s}\oplus \varphi(\mathfrak{h})]=0$), and $\mathfrak{s}$ is semisimple.
\medskip

\begin{lemma}\label{comni2}
If a linear space $\mathfrak{q}\subset \mathfrak{r(g)}$ is $\ad(\mathfrak{h})$-invariant, then
$[\mathfrak{h},\mathfrak{q}]\subset \mathfrak{q}\cap [\mathfrak{g},\mathfrak{r(g)}]  \subset \mathfrak{q}\cap\mathfrak{n(g)}$.
In particular, $\mathfrak{q}\cap \mathfrak{n(g)}=0$ or $\mathfrak{q}\cap[\mathfrak{g},\mathfrak{r(g)}]=0$ implies
$[\mathfrak{h},\mathfrak{q}]=0$.
\end{lemma}

\begin{proof} It suffices to use  $[\mathfrak{g},\mathfrak{r(g)}]\subset \mathfrak{n(g)}$.
\end{proof}
\smallskip

\begin{lemma}\label{h.1}
We can choose an $\Ad(H)$-invariant complement $\mathfrak{h}_1$ to
$\operatorname{Ker}(\psi)$ in $\mathfrak{h}$ such that $\mathfrak{h}_1 \cap \mathfrak{s}=\mathfrak{h} \cap \mathfrak{s}$,
in particular, $[\mathfrak{h},\mathfrak{h}]=[\mathfrak{h}_1,\mathfrak{h}_1]$.
Let $\mathfrak{h}_2$ be an $\Ad(H)$-invariant complement to $\mathfrak{h}_1 \cap \mathfrak{s}$ in such $\mathfrak{h}_1$.
Then  $\mathfrak{h}_2\subset \mathfrak{c(h)}$, i.~e. $\mathfrak{h}_2$ is central in $\mathfrak{h}$,
the Lie algebras $\mathfrak{h}_2$, $\varphi(\mathfrak{h}_2)$, and $\psi(\mathfrak{h}_2)$ are pairwise isomorphic, and the following
decompositions hold:
$$
\mathfrak{h}=\mathfrak{h}_2 \oplus (\mathfrak{h} \cap \mathfrak{r(g)}) \oplus(\mathfrak{h} \cap \mathfrak{s}), \quad
\varphi(\mathfrak{h})=\varphi(\mathfrak{h}_2)\oplus(\mathfrak{h} \cap \mathfrak{r(g)}),\quad
\psi(\mathfrak{h})=\psi(\mathfrak{h}_2)\oplus(\mathfrak{h} \cap \mathfrak{s}).
$$
\end{lemma}

\begin{proof}
The space $\mathfrak{h} \cap \mathfrak{s}$ is $\Ad(H)$-invariant and $(\mathfrak{h} \cap \mathfrak{s})\cap \operatorname{Ker}(\psi)=0$.
Let us $\mathfrak{l}$ be any $\Ad(H)$-invariant complement to $(\mathfrak{h} \cap \mathfrak{s})\oplus \operatorname{Ker}(\psi)$ in $\mathfrak{h}$.
Then we can take $\mathfrak{h}_1:=(\mathfrak{h} \cap \mathfrak{s})\oplus \mathfrak{l}$. Since
$[\mathfrak{h},\mathfrak{h}]\subset \mathfrak{s}$ we see that
$\mathfrak{l}$ is in the center of $\mathfrak{h}$. From this we get $[\mathfrak{h},\mathfrak{h}]=[\mathfrak{h}_1,\mathfrak{h}_1]$.

It is clear that $\operatorname{Ker}(\psi)$ and $\mathfrak{h}_1$ are ideals in $\mathfrak{h}$, hence
$\mathfrak{h}=\operatorname{Ker}(\psi)\oplus \mathfrak{h}_1$ as Lie algebras.
Obviously, the restriction of $\psi$ to $\mathfrak{h}_1$ is one-to-one.  Moreover,
$\operatorname{Ker}(\psi) \subset \varphi(\mathfrak{h})$ and
$[\operatorname{Ker}(\psi), \mathfrak{h}]\subset [\varphi(\mathfrak{h}), \varphi(\mathfrak{h})\oplus \mathfrak{s}]=0$, hence $\operatorname{Ker}(\psi)$
is central in $\mathfrak{h}$. Note also that the Lie algebras $\mathfrak{h}_1$ and $\psi(\mathfrak{h})$ are isomorphic (under $\psi$).

It is clear that the subalgebra $\mathfrak{h}_1 \cap \mathfrak{s}= \mathfrak{h} \cap \mathfrak{s}$ is $\Ad(H)$-invariant.
Since $[\mathfrak{h}_1,\mathfrak{h}_1]\subset \mathfrak{h}_1$, $[\mathfrak{h},\mathfrak{h}]=[\mathfrak{h}_1,\mathfrak{h}_1]$, and
$[\mathfrak{h},\mathfrak{h}]\subset \mathfrak{s}$, we have
$[\mathfrak{h},\mathfrak{h}]=[\mathfrak{h}_1,\mathfrak{h}_1] \subset
\mathfrak{h}_1 \cap \mathfrak{s}$.
In particular, $[\mathfrak{h}_2, \mathfrak{h}]\subset \mathfrak{h}_1 \cap \mathfrak{s}$.
Since $\mathfrak{h}_2$ is an $\Ad(H)$-invariant complement to $\mathfrak{h}_1 \cap \mathfrak{s}$ in $\mathfrak{h}_1$, then
$[\mathfrak{h}_2, \mathfrak{h}]\subset \mathfrak{h}_2$, hence,
$\mathfrak{h}_2\subset \mathfrak{c(h)}$, i.~e. $\mathfrak{h}_2$ is central in $\mathfrak{h}$.
It is clear also that $\mathfrak{h}_2$, $\varphi(\mathfrak{h}_2)$, and $\psi(\mathfrak{h}_2)$ are pairwise isomorphic.
Indeed, $\psi$ is one-to-one even on $\mathfrak{h}_1$, but if $X\in \mathfrak{h}_2$ and $\varphi(X)=0$,
then $x = \psi(X)\in \mathfrak{h} \cap \mathfrak{r(g)}$, that is also impossible. Hence, we get
$\mathfrak{h}=\mathfrak{h}_2 \oplus (\mathfrak{h} \cap \mathfrak{r(g)}) \oplus(\mathfrak{h} \cap \mathfrak{s})$,
$\varphi(\mathfrak{h})=\varphi(\mathfrak{h}_2)\oplus(\mathfrak{h} \cap \mathfrak{r(g)})$, and
$\psi(\mathfrak{h})=\psi(\mathfrak{h}_2)\oplus(\mathfrak{h} \cap \mathfrak{s})$.
\end{proof}
\medskip

Now we can describe some useful types of reductive decomposition (\ref{reductivedecomposition}), using the above constructions.
The subspaces $\operatorname{Im}(\psi)=\psi(\mathfrak{h})$ and $\operatorname{Ker}(\psi)=\mathfrak{h} \cap \mathfrak{r(g)}$  in $\mathfrak{g}$
are $\Ad(H)$-invariant.
Now, consider any $\Ad(H)$-invariant complement $\mathfrak{m}_1$ to $\operatorname{Ker}(\psi)$ in $\mathfrak{r(g)}$ and
any $\Ad(H)$-invariant complement $\mathfrak{m}_2$ to $\operatorname{Im}(\psi)$ in $\mathfrak{s}$.
{\it Then $\mathfrak{m}:=\mathfrak{m}_1\oplus \mathfrak{m}_2$ is
an $\Ad(H)$-invariant complement to $\mathfrak{h}$ in~$\mathfrak{g}$ according to} Lemma \ref{h.1}.
\medskip

Sometimes,  it would be better (for some technical reasons) to deal with a case, when $\mathfrak{n(g)}\subset \mathfrak{m}_1$.
It is possible by the following reason.
The Killing form $B$ is negative on $\mathfrak{h}$ and zero on $\mathfrak{n(g)}$, therefore, $\mathfrak{h}\cap \mathfrak{n(g)}=0$.
On the other hand, $\mathfrak{n(g)}$ is a characteristic ideal in $\mathfrak{g}$, hence is invariant with respect to $\Ad(H)$.
Now, we can define an $\Ad(H)$-invariant complement $\mathfrak{u}$ to  $\mathfrak{n(g)} \oplus \operatorname{Ker}(\psi)$ in $\mathfrak{r(g)}$ and put
$\mathfrak{m}_1:=\mathfrak{n(g)} \oplus \mathfrak{u}$. Since $[\mathfrak{g},\mathfrak{r(g)}]\subset \mathfrak{n(g)}$ we get $[\mathfrak{h},\mathfrak{u}]=0$.

\begin{remark}\label{rem1}
There is another natural way to choose $\mathfrak{m}_1$ such that $\mathfrak{n(g)}\subset \mathfrak{m}_1$: let us consider
$\mathfrak{m}_1^{\prime}:=\{X\in \mathfrak{g}\,|\, B(X,Y)=0\,\,\,\forall\, Y\in \mathfrak{s}+\mathfrak{h}\}$. It is clear that
$\mathfrak{n(g)}\subset \mathfrak{m}_1^{\prime}\subset\mathfrak{r(g)}$ by~(\ref{nilid}) because
$[\mathfrak{g},\mathfrak{g}]\subset\mathfrak{n(g)}\oplus \mathfrak{s}$ and
$$
\mathfrak{r(g)}=\{X\in \mathfrak{g}\,|\, B(X,Y)=0\,\,\, \forall\, Y\in [\mathfrak{g},\mathfrak{g}]\}=
\{X\in \mathfrak{g}\,|\, B(X,Y)=0\,\,\, \forall\, Y\in \mathfrak{s}\},
$$
because $B(\mathfrak{n(g)},\mathfrak{g})=0$.

Further, $\mathfrak{s}+\mathfrak{h}$ is $\Ad(H)$-invariant, hence, $\mathfrak{m}_1^{\prime}$ has the same property.
The Killing form $B$ is negatively definite on $\operatorname{Ker}(\psi)=\mathfrak{h} \cap \mathfrak{r(g)}$, therefore,
$\mathfrak{m}_1^{\prime}\cap \operatorname{Ker}(\psi)=0$.
Now let us consider $\mathfrak{m}_1^{\prime \prime}$, an $\Ad(H)$-invariant complement to $\mathfrak{m}_1^{\prime} \oplus \operatorname{Ker}(\psi)$ in
$\mathfrak{r(g)}$. Then we can put $\mathfrak{m}_1:=\mathfrak{m}_1^{\prime}\oplus \mathfrak{m}_1^{\prime \prime}$.
It is clear that $\mathfrak{m}_1^{\prime}=\{X\in \mathfrak{r(g)}\,|\, B(X,Y)=0\, \forall\, Y\in \varphi(\mathfrak{h})\}$.

Let $\mathfrak{l}$ be the $B$-orthogonal complement to $\operatorname{Ker}(\psi)=\mathfrak{h} \cap \mathfrak{r(g)}$ in $\mathfrak{h}$.
Note that for any $Y\in \operatorname{Ker}(\psi)$ and any $X\in \mathfrak{l}$, we get $0=B(Y,X)=B(Y,\psi(X)+\varphi(X))=B(Y,\varphi(X))$, because
$B(\mathfrak{r(g)},\mathfrak{s})=0$.
\end{remark}

\begin{example}\label{nilradspesex}
It should be noted that there are reductive decompositions (\ref{reductivedecomposition}) such that $\mathfrak{n(g)}\not\subset \mathfrak{m}$.
Let us consider the homogeneous space $G/H=U(2)/S^1=(SU(2)\times S^1)/S^1$,
where the isotropy group is embedded as follows:
$$
H=S^1=\diag(S^1)\subset S^1\times S^1 \subset SU(2)\times S^1,
$$
where the embedding of the second multiple $S^1$
in the product $S^1 \times S^1$ is identical, and the embedding of the first multiple
is defined by some fixed embedding
$i:S^1 \rightarrow SU(2)$.
Since all circles (maximal tori) in the group $SU(2)$ are pairwise conjugate
under the adjoint action of $SU(2)$, then
we get a unique homogeneous space (up to homogeneous spaces morphism) diffeomorphic to $S^3$.
In the Lie algebra $\mathfrak{g}=su(2) \oplus \mathbb{R}$, we choose vectors $e_1,\dots,e_4$ such that
$e_i \in su(2)$ for $1\leq i \leq 3$, $e_4 \in \mathbb{R}$,
$\mathfrak{h}=\operatorname{Lie}(H)=\Lin\{e_3+e_4\}$,
$[e_1,e_2]=e_3$, $[e_2,e_3]=e_1$, $[e_3,e_1]=e_2$.
Let us fix an $\Ad(G)$-invariant inner product $\langle \cdot,\cdot \rangle$ on $\mathfrak{g}$  such that
the vectors $e_i$ ($i=1,\dots, 4$) form an orthonormal basis  with respect to this inner product.
For any  $\alpha >0$ consider
a non-degenarate $\Ad(G)$-invariant quadratic form
$$
(\cdot,\cdot)_{\alpha}=\langle \cdot,\cdot \rangle|_{su(2)}+ \alpha \langle \cdot ,\cdot \rangle |_{\mathbb{R}}
$$
on $\mathfrak{g}$. Now, consider the orthogonal complement $\mathfrak{m}$ to $\mathfrak{h}$ in $\mathfrak{g}$
with respect to $(\cdot,\cdot)_{\alpha}$, that defines a reductive decomposition of $\mathfrak{g}$. The restriction
$(\cdot,\cdot)_{\alpha}$ to $\mathfrak{m}$ defines a normal homogeneous, hence, a geodesic orbit
metric, on the space $G/H=S^3$ for all positive $\alpha$, see a more detailed discussion e.~g. in \cite{Nik2013}.
Note that $\mathfrak{n(g)}=\Lin\{e_4\} \not\subset \mathfrak{m}$ for any $\alpha>0$, since $(\mathfrak{h},e_4)_{\alpha}\neq 0$.

It is clear that there are similar reductive decompositions for all compact homogeneous spaces $G/H$ with non-semisimple groups $G$.
\end{example}

Sometimes it is reasonable to use the following $\Ad(H)$-invariant complement to $\mathfrak{h}$ in $\mathfrak{g}$:
\begin{equation}\label{compkill1}
\mathfrak{m}=\{X\in \mathfrak{g}\,|\, B(X,Y)=0\, \, \mbox{ for all } Y\in \mathfrak{h}\}.
\end{equation}
Since $B$ is negatively definite on $\mathfrak{h}$, $\mathfrak{m}$ is a complement to $\mathfrak{h}$ in $\mathfrak{g}$.
Note also that $\mathfrak{n(g)}\subset\mathfrak{m}$ due to (\ref{nilid}). Advantages of this reductive decomposition follow from the
next lemma.

\begin{lemma}\label{h.2}
If $H$ is a compact subgroup in $G$, then
$\mathfrak{m}:=\{X\in \mathfrak{g}\,|\,B(X,\mathfrak{h})=0\}$ is an $\Ad(H)$-invariant complement to $\mathfrak{h}$ in $\mathfrak{g}$.
Consider the operator $A:\mathfrak{m} \rightarrow \mathfrak{m}$ such that $B(X,Y)=(AX,Y)$, $X,Y\in \mathfrak{m}$.
Then the eigenspaces of $A$ are $\Ad(H)$-invariant and
pairwise orthogonal both respect to $B$ and $(\cdot,\cdot)$.
Moreover, there is a {\rm(}possibly, non-unique{\rm)} $B$-orthogonal and $(\cdot,\cdot)$-orthogonal decomposition
$\mathfrak{m}=\oplus_{i=1}^s \mathfrak{p}_i$,
where $\mathfrak{p}_i$ are $\Ad(H)$-invariant and $\Ad(H)$-irreducible {\rm(}$\ad(\mathfrak{h})$-invariant and
$\ad(\mathfrak{h})$-irreducible{\rm)} submodules. Moreover,
$[\mathfrak{p}_i,\mathfrak{p}_j]\subset \mathfrak{m}$, $i\neq j$, for any such decomposition.
\end{lemma}

\begin{proof}
Since the Killing form $B$ is negatively definite on $\mathfrak{h}$, the first assertion is obvious. Let us prove the second assertion.
Clear that the operator $A$ is $\Ad(H)$-equivariant and symmetric (with respect to $(\cdot,\cdot)$), in particular, the eigenspaces of $A$ are $\Ad(H)$-invariant.
Let $A_{\alpha}$ and $A_{\beta}$ be two eigenspaces of $A$ with the eigenvalues $\alpha\neq \beta$.
Consider any $X\in A_{\alpha}$ and $Y\in A_{\beta}$ and suppose that $\beta \neq 0$. Then we have
$$\beta(X,Y)=(X,AY)=B(X,Y)=(AX,Y)=\alpha(X,Y)$$
that implies
$(X,Y)=B(X,Y)=0$. Hence,  the eigenspaces of $A$
are pairwise orthogonal with respect to $B$ and $(\cdot,\cdot)$. Further, one can decompose every such eigenspace into
a direct $(\cdot,\cdot)$-orthogonal sum of $\Ad(H)$-invariant and $\Ad(H)$-irreducible (or even $\ad(\mathfrak{h})$-invariant and
$\ad(\mathfrak{h})$-irreducible) submodules.
Since $B$ is a multiple of $(\cdot,\cdot)$ on every eigenspace of $A$, we get the required decomposition.
Finally if $[\mathfrak{p}_i,\mathfrak{p}_j]\not\subset \mathfrak{m}$, then $B([X,Y],Z)\neq 0$ for some
$X\in \mathfrak{p}_i$, $Y\in \mathfrak{p}_j$, $Z\in \mathfrak{h}$. This implies $B(Y,[Z,X])\neq 0$ and $\mathfrak{p}_i$ is not
$\ad(\mathfrak{h})$-invariant, that is impossible.
\end{proof}

\begin{remark}\label{eigkill1}
Note that $A_0$ is an ideal of $\mathfrak{g}$. Indeed, $A_0=\{X\in \mathfrak{g}\,|\, B(X,\mathfrak{g})=0\}$,
the equality $B(A_0,\mathfrak{g})=0$ implies $B(A_0,[\mathfrak{g},\mathfrak{g}])=0$ and $B([\mathfrak{g}, A_0],\mathfrak{g})=0$.
Moreover, $\mathfrak{n(g)}\subset A_0 \subset \mathfrak{r(g)}$ according to (\ref{nilid}).
\end{remark}

It is naturally to consider the reductive complement (\ref{compkill1}) in order to study various problems related with homogeneous spaces.
However, sometimes this complement is not the most helpful.

\begin{example}
Let us consider Ledger~--~Obata spaces, i.~e. the spaces $G/H=F^m/\diag(F)$, $m\geq 2$, where $F$ is a connected compact simple Lie group,
$G=F^m=F\times F\times\cdots\times F$
($m$ factors), and $H=\diag(F)=\{(X,X,\dots,X)|X\in F\}$. Note that for $m=2$ we get irreducible symmetric spaces.

Let $\mathfrak{f}$ be the Lie algebra of $F$, then
$\mathfrak{g}:=m\mathfrak{f}$
and $\mathfrak{h}:=\diag(\mathfrak{f})$.
There are several $\Ad({H})$-invariant complements to
$\mathfrak{h}=\diag(\mathfrak{f})$ in $\mathfrak{g}=m \mathfrak{f}$. Indeed, it is easy to see that
every irreducible $\Ad({H})$-invariant module in $\mathfrak{g}$ has the form
$\left\{(\alpha_1 X, \alpha_2 X,\dots, \alpha_{m} X)\subset \mathfrak{g}\,|\, X \in \frak{f}\right\}$
for some fixed $\alpha_i \in \mathbb{R}$, see details e.~g. in \cite{CNN2017}. Every such modules are pairwise isomorphic with respect to $\Ad({H})$.
This allows to construct various $\Ad({H})$-invariant complements to
$\mathfrak{h}=\diag(\mathfrak{f})=\left\{(X,X,\dots,X)\,|\, X \in \mathfrak{f}\right\}$.

The most natural complement to $\mathfrak{h}$ in $\mathfrak{g}$ is $\mathfrak{m}$ defined as in (\ref{compkill1}).
It is easy to see that
$$
\mathfrak{m}=\left\{\operatorname{Lin}\bigl((\alpha_1 X, \alpha_2 X,\dots, \alpha_{m} X)\bigr)
\subset \mathfrak{g}\,|\, X \in \frak{f},\, \sum_{i=1}^m \alpha_i =0\right\}.
$$

Now, let $\mathfrak{g}_k$, $k=1,\dots,m$, be an ideal in $m \mathfrak{f}$ of the form
$\mathfrak{f}\oplus\cdots \oplus \mathfrak{f}\oplus 0 \oplus \mathfrak{f}\oplus\cdots \oplus \mathfrak{f}$, where $0$ is instead of $k$-th copy of
$\mathfrak{f}$ in $m \mathfrak{f}$.
We can choose any of $\mathfrak{g}_k$ as a  $\Ad(\diag(F))$-invariant complement to $\mathfrak{h}$ in $\mathfrak{g}=m \mathfrak{f}$.
Such choice proves also that the space of invariant Riemannian metric on $G/H=F^m/\diag(F)$ could be consider as the space of $\Ad(\diag(F))$-invariant
metrics on the Lie group $F^{m-1}$.
Note that the  complement $\mathfrak{g}_{m}$ was used instead of $\mathfrak{m}$ in \cite{CNN2017}, that helped to classify all Einstein invariant metrics
on the spaces $G/H=F^m/\diag(F)$ for $m\leq 4$.
\end{example}

\smallskip

Finally, we describe a possible choice of a reductive complement $\mathfrak{m}$ with using of maximal compactly embedded subalgebras in $\mathfrak{g}$.
\smallskip

Now, for a compact subgroup $H$ of the Lie group $G$ we can choose a maximal compactly embedded subgroup $K$ in $G$ such that $H \subset K$,
since all maximal compactly embedded subgroups of $G$ are conjugate. Now, we can choose a Levi subalgebra $\mathfrak{s}$ in
$\mathfrak{g}$, which is invariant under $\Ad_G(K)$. Then we get

\begin{lemma}[Lemma 14.3.3 in \cite{HilNeeb}]\label{goodlevi1}
For every maximal compactly embedded subalgebra $\mathfrak{k}$ of $\mathfrak{g}$,
there exists an $\Ad_G(K)$-invariant Levi decomposition $\mathfrak{g} = \mathfrak{r(g)}\rtimes \mathfrak{s}$ with the following properties:

{\rm 1)} $[\mathfrak{k}, \mathfrak{s}] \subset \mathfrak{s}$;

{\rm 2)} $[\mathfrak{k}\cap \mathfrak{r(g)}, \mathfrak{s}] = 0$;

{\rm 3)} $\mathfrak{k} = (\mathfrak{k}\cap \mathfrak{r(g)}) \oplus (\mathfrak{k}\cap \mathfrak{s})$;

{\rm 4)} $[\mathfrak{k},\mathfrak{k}] \subset \mathfrak{s}$;

{\rm 5)} $\mathfrak{k} \cap \mathfrak{s}$ is a maximal compact subalgebra in $\mathfrak{s}$.
\end{lemma}
It is easy to see also that $\mathfrak{k} \cap \mathfrak{r(g)}$ is a maximal compactly embedded subalgebra in $\mathfrak{r(g)}$, since all
maximal compactly embedded subalgebras of $\mathfrak{g}$ are conjugated in $\operatorname{Aut}(\mathfrak{g})$ and $\mathfrak{r(g)}$ is a characteristic ideal in
$\mathfrak{g}$. See also Theorem 2 in \cite{Wust} for Lie groups.
Now, in order to choose an $\Ad(H)$-invariant complement $\mathfrak{m}$ to $\mathfrak{h}$ in $\mathfrak{g}$ one need to choose some
$\Ad(H)$-invariant complement $\mathfrak{m}_1$ to $\mathfrak{h}$ in $\mathfrak{k}$, some
$\Ad(K)$-invariant complement $\mathfrak{m}_2$ to $\mathfrak{k} \cap \mathfrak{s}$ in $\mathfrak{s}$, and
some $\Ad(K)$-invariant complement $\mathfrak{m}_3$ to $\mathfrak{k}\cap \mathfrak{r(g)}$ in $\mathfrak{r(g)}$, and put
$\mathfrak{m}:=\mathfrak{m}_1\oplus \mathfrak{m}_2\oplus\mathfrak{m}_3$.

\section{On totally geodesic submanifolds}

Let us consider an effective homogeneous Riemannian space $(M=G/H, \rho)$ with $G$-invariant metric $\rho$ and some $\Ad(H)$-invariant decomposition
$\mathfrak{g}=\mathfrak{h}\oplus \mathfrak{m}$, as it is explained in the previous section.
We identify  elements $X,Y \in  \mathfrak{g}$ with Killing vector fields on $M=G/H$, the
$\mathfrak{m}$ with the tangent space at the point $o = eH$, and the metric $\rho$ with an inner product $(\cdot,\cdot)$ on $\mathfrak{m}$ as usual.
Then the covariant derivative $\nabla_XY$, $X,Y \in  \mathfrak{m}$, at the point $o = eH$ is given by

\begin{equation}\label{connec}
\nabla_XY(o)=- \frac{1}{2}\, [X,Y]_{\mathfrak{m}}+U(X,Y),
\end{equation}
where the bilinear symmetric map $U:\mathfrak{m}\times \mathfrak{m}
\rightarrow \mathfrak{m}$ is given by
\begin{equation}\label{connec1}
2(U(X,Y),Z)=( [Z,X]_{\mathfrak{m}},Y)+(X,[Z,Y]_{\mathfrak{m}})
\end{equation}
for any $X,Y, Z\in \mathfrak{m}$, where $V_{\mathfrak{m}}$ is the
$\mathfrak{m}$-part of a vector $V \in \mathfrak{g} $ \cite[Proposition 7.28]{Bes}.

We have the following important property of totally geodesic submanifolds of a geodesic orbit Riemannian manifold.

\begin{prop}[\cite{BerNik}, Theorem 11]\label{tot}
Every closed totally geodesic submanifold of a  Riemannian manifold
with  homogeneous geodesics is a manifold with  homogeneous
geodesics.
\end{prop}
\medskip

Note that  a totally geodesic submanifold $M'$ of a geodesic orbit Riemannian space $M=G/H$ could have
the isometry group that is not a subgroup of $G$.

\begin{example}\label{spheretotge}  The round sphere $S^{2n-1}$
can be realized as the homogeneous Riemannian space $U(2n)/U(2n-1)$, but any its totally geodesic submanifold $S^{2n-2}$
admits only one connected transitive isometry group $SO(2n-1)$, that is not a subgroup of $U(2n)$.
\end{example}

It is possible to produce various generalizations of this example. This effects explains some troubles with the study of totally geodesic submanifolds
of a given geodesic orbit Riemannian space.

For a subspace $\mathfrak{p} \subset \mathfrak{m}$ and $U \in \mathfrak{m}$ we denote
by $U_{\mathfrak{p}}$ the $(\cdot,\cdot)$-orthogonal projection of $U$ to $\mathfrak{p}$.
The restriction of the adjoint
endomorphism $\ad (X)$, $X \in \mathfrak{g}$, to a subspace $\mathfrak{p} \subset
\mathfrak{m}$  will be  denoted by
$\ad_X^{\mathfrak{p}}$, i.~e. $\ad_X^{\mathfrak{p}}(Y)=([X,Y]_{\mathfrak{m}})_{\mathfrak{p}}$, $Y\in\mathfrak{p}$.
The notation $\mathfrak{p}^\perp$ means the
$(\cdot,\cdot)$-orthogonal complement to $\mathfrak{p}\subset \mathfrak{m}$ in $\mathfrak{m}$.

\begin{definition} A subspace $\mathfrak{p} \subset \mathfrak{m}$
is called {\it totally geodesic} if it is the tangent space  at $o$
of a totally geodesic orbit $Ko \subset G/H =M$ of a subgroup $K
\subset G$.
\end{definition}

Note that the property to be a totally geodesic subspace depends on a reductive decomposition (\ref{reductivedecomposition}).
From (\ref{connec}) and (\ref{connec1}) we easily get the following  result.

\begin{prop}[\cite{AN}]\label{totallygeodesicProp} Let $(M =G/H,\rho)$ be a homogeneous Riemannian space with the reductive decomposition
{\rm (\ref{reductivedecomposition})}.
Then a subspace $\mathfrak{p} \subset \mathfrak{m}$
is totally geodesic if  and only if the following two conditions hold:

{\rm 1)} $\mathfrak{p}$ generates a  subalgebra  of the form $\mathfrak{k} =
\mathfrak{h}' \oplus \mathfrak{p}$, where $\mathfrak{h}'$ is a
subalgebra of $\mathfrak{h}$ {\rm(}this is equivalent to $[\mathfrak{p},\mathfrak{p}] \subset \mathfrak{h} \oplus \mathfrak{p}$,
and $\mathfrak{h}'$ could be chosen as
minimal subalgebra of $\mathfrak{h}$, including $[\mathfrak{p},\mathfrak{p}]_{\mathfrak{h}}${\rm)};

{\rm 2)} the endomorphism $\ad_Z^{\mathfrak{p}} \in \End(\mathfrak{p})$
for  $Z \in \mathfrak{p}^{\perp}$ is $(\cdot,\cdot)$-skew-symmetric or,
equivalently,
$U(\mathfrak{p}, \mathfrak{p}) \subset \mathfrak{p}$.
\end{prop}

The following result is well known.

\begin{lemma} \label{skew-symmetryLemma}
Let $(M =G/H,\rho)$ be a geodesic orbit Riemannian space with the reductive decomposition
{\rm (\ref{reductivedecomposition})} and $\mathfrak{m} = \mathfrak{p}\oplus \mathfrak{q}$
is a $(\cdot,\cdot)$-orthogonal decomposition. Then
$$
U(\mathfrak{p},\mathfrak{p} ) \subset \mathfrak{p},\quad
U(\mathfrak{q}, \mathfrak{q}) \subset \mathfrak{q},
$$
and the endomorphisms
$\ad_{\mathfrak{p}}^{\mathfrak{q}},\,\,\ad_{\mathfrak{q}}^{\mathfrak{p}}
$ are skew-symmetric with respect to $(\cdot,\cdot)$.
\end{lemma}

\begin{proof} Clear, that the decomposition $\mathfrak{m} = \mathfrak{p}\oplus \mathfrak{q}$ is $\Ad(H)$-invariant.
For $X \in \mathfrak{p}$, $Y \in \mathfrak{q} $ we
have
$$
0= \left([Y + Z, X]_{\mathfrak{m}},Y \right) = - \left(\ad (X) Y,Y
\right) =- \left( U(Y,Y), X \right) ,
$$
where $Z$ is as in Lemma \ref{GO-criterion}. This shows that
$\ad_{X}^{\mathfrak{q}}$ is skew-symmetric  and
$U(\mathfrak{q},\mathfrak{q} ) \subset \mathfrak{q}$.
\end{proof}
\smallskip

Now, from Lemma \ref{skew-symmetryLemma} we get

\begin{prop}\label{totgeodet}
Let $(M=G/H, \rho)$ be a geodesic orbit Riemannian space with the reductive decomposition~{\rm(\ref{reductivedecomposition})}, then the following assertions hold:

{\rm 1)} If  a subspace $\mathfrak{p} \subset \mathfrak{m}$
generates a Lie subalgebra $\mathfrak{k} =
\mathfrak{h}' \oplus \mathfrak{p}$ of $\mathfrak{g}$, where $\mathfrak{h}'=\mathfrak{k} \cap \mathfrak{h}$ is a
subalgebra of $\mathfrak{h}$ and $K \subset G$ is the corresponding subgroup,
then the homogeneous space $K/K\cap H$ with the induced Riemannian metric
is a totally geodesic submanifold of $(M=G/H, \rho)$.

{\rm 2)} If $\mathfrak{h}'+C_{\mathfrak{h}}(\mathfrak{p})=\mathfrak{h}$,
where $C_{\mathfrak{h}}(\mathfrak{p})$ is the centralizer of $\mathfrak{p}$ in $\mathfrak{h}$,
then $K/K\cap H$ with the induced Riemannian metric is a geodesic orbit Riemannian space itself.
In particular, any  connected subgroup $K
\subset G$ which contains $H$ has the totally geodesic orbit $P = Ko = K/H$ which is a geodesic orbit Riemannian space
{\rm(}with respect to the induced metric{\rm)}.
\end{prop}

\begin{proof}
Clear that $K/K\cap H$ with the induced Riemannian metric is totally geodesic in $(M=G/H, \rho)$ by
Lemma \ref{skew-symmetryLemma} and  Proposition \ref{totallygeodesicProp}.

Now, consider any $X\in \mathfrak{p}$. By Lemma \ref{GO-criterion} there is $Z \in \mathfrak{h}$ such that
$\left([Y + Z, X]_{\mathfrak{m}},Y \right)=0$
for any $Y\in \mathfrak{m}$. It suffices to prove that we can choose such $Z$ from $\mathfrak{h}'$
(then we will have the following:
for any $X\in \mathfrak{p}$ there is $Z \in \mathfrak{h}'$ such that
$\left([Y + Z, X]_{\mathfrak{p}},Y \right)=0$ for any $Y\in \mathfrak{p}$,
that means the geodesic orbit property for the homogeneous space $K/K\cap H$ with the induced Riemannian metric).

Let $C_{\mathfrak{h}}(X)$ the centralizer of $X$ in $\mathfrak{h}$. It is easy to see that we can take any $Z+U$ instead of $Z$ for any $U\in C_{\mathfrak{h}}(X)$
in Lemma \ref{GO-criterion}. Since $\mathfrak{h}'+C_{\mathfrak{h}}(\mathfrak{p})=\mathfrak{h}$, then
$\mathfrak{h}'+C_{\mathfrak{h}}(X)=\mathfrak{h}$, hence, we may choose $U\in C_{\mathfrak{h}}(X)$ such that
$Z+U\in \mathfrak{h}'$, that proves the second assertion.
\end{proof}
\medskip

It should be noted that the equality $\mathfrak{h}'+C_{\mathfrak{h}}(\mathfrak{p})=\mathfrak{h}$ (see Proposition \ref{totgeodet})
is wrong in general, see Example \ref{examplenil} below.

\begin{remark}\label{rem5} In Proposition \ref{totgeodet}, the homogeneous space $K/K\cap H$ with the induced Riemannian metric
is always totally geodesic submanifold of $(M=G/H, \rho)$, hence it is a geodesic orbit Riemannian
manifold by Proposition \ref{tot}. But we could not point out
(in general) the connected isometry group of this manifold.  On the other hand, the subgroup $K'$
corresponded to the subalgebra $\mathfrak{k}'=\mathfrak{h}\oplus \mathfrak{p}$, act on this manifolds transitively
(but not necessarily effectively). Hence, the space $K'/H$ with the induced metric is geodesic orbit.
Note also that instead of $\mathfrak{h}$ in $\mathfrak{k}'$ one can take any subalgebra $\mathfrak{h}''$ in $\mathfrak{h}$ such that
$\mathfrak{h}'\subset \mathfrak{h}''$ and $C_{\mathfrak{h}}(\mathfrak{p})+\mathfrak{h}''=\mathfrak{h}$, where
$C_{\mathfrak{h}}(\mathfrak{p})$ is the centralizer of $\mathfrak{p}$ in $\mathfrak{h}$.
\end{remark}

\begin{remark}\label{rem4} If $G$ is compact or semisimple and $\mathfrak{m}$ is orthogonal to $\mathfrak{h}$ with respect to some $\Ad(G)$-invariant nondegenerate
(not necessarily positive definite)
inner product  $\langle \cdot, \cdot \rangle$, then the equality $\mathfrak{h}'+C_{\mathfrak{h}}(\mathfrak{p})=\mathfrak{h}$ holds.
Indeed, the equality $0=\langle U, \mathfrak{h}'\rangle$  for $U \in \mathfrak{h}$ implies $0=\langle U,[\mathfrak{p}, \mathfrak{p}]_{\mathfrak{h}}\rangle =
\langle U, [\mathfrak{p}, \mathfrak{p}]\rangle=\langle \mathfrak{p},[U, \mathfrak{p}]\rangle$, hence
$U \in C_{\mathfrak{h}}(\mathfrak{p})$. Note that for a general $G$ this is not the case.
\end{remark}

\begin{prop}\label{irred1} Let $(M=G/H, \rho)$ be a homogeneous Riemannian space with the reductive decomposition~{\rm(\ref{reductivedecomposition})}.
If $\mathfrak{p}\subset \mathfrak{m}$ is an $\ad(\mathfrak{h})$-invariant submodule, then
the subspace $C_{\mathfrak{h}}(\mathfrak{p}):=\{V\in \mathfrak{h}\,|\, [V,\mathfrak{p}]=0\}$ is an ideal in $\mathfrak{h}$.
If in addition $\mathfrak{p}$ is such that $[\mathfrak{p},\mathfrak{p}]\subset \mathfrak{p}\oplus \mathfrak{h}$, then
$[\mathfrak{p},\mathfrak{p}]_{\mathfrak{h}}=\operatorname{Lin}\{[X,Y]_{\mathfrak{h}}\,|\, X,Y\in \mathfrak{p}\}$ is an ideal in $\mathfrak{h}$.
Moreover, $B(C_{\mathfrak{h}}(\mathfrak{p}),[\mathfrak{p},\mathfrak{p}]_{\mathfrak{h}})=0$, where $B$ is the Killing form of the Lie algebra $\mathfrak{g}$,
under any of the following conditions:

{\rm 1)} $B(\mathfrak{p},\mathfrak{h})=0$;

{\rm 2)} $C_{\mathfrak{h}}(\mathfrak{p})$ is semisimple;

{\rm 3)} $\mathfrak{p}$ is $\ad(\mathfrak{h})$-irreducible;

{\rm 4)} $\mathfrak{p}$ has no 1-dimensional $\ad(\mathfrak{h})$-invariant submodule.
\end{prop}

\begin{proof}
For any $V \in \mathfrak{h}$ and any $U\in C_{\mathfrak{h}}(\mathfrak{p})$, we get
$$
[[U,V],\mathfrak{p}]\subset [[U,\mathfrak{p}],V]+[U,[V,\mathfrak{p}]]\subset [U,\mathfrak{p}]=0,
$$
which proves the first assertion.

The second assertion follows from
$$
[U,[X,Y]_{\mathfrak{h}}]=[U,[X,Y]]_{\mathfrak{h}}=[[U,X],Y]]_{\mathfrak{h}}+[X,[U,Y]]_{\mathfrak{h}}\subset
[\mathfrak{p},\mathfrak{p}]_{\mathfrak{h}},
$$
where $U\in \mathfrak{h}$ and $X,Y \in \mathfrak{p}$.

In order to prove the third assertion point out that
$B([\mathfrak{g},\mathfrak{p}],C_{\mathfrak{h}}(\mathfrak{p}))=B(\mathfrak{g},[\mathfrak{p},C_{\mathfrak{h}}(\mathfrak{p})]]=0$.
In particular, $B([\mathfrak{p},\mathfrak{p}],C_{\mathfrak{h}}(\mathfrak{p}))=0$. Therefore,
$B(\mathfrak{p},\mathfrak{h})=0$ implies $B(\mathfrak{p},C_{\mathfrak{h}}(\mathfrak{p}))=0$ and
the condition 1) implies $B(C_{\mathfrak{h}}(\mathfrak{p}),[\mathfrak{p},\mathfrak{p}]_{\mathfrak{h}})=0$.

Note that
$B(\mathfrak{p},[\mathfrak{g},C_{\mathfrak{h}}(\mathfrak{p})])=-B([\mathfrak{g},\mathfrak{p}],C_{\mathfrak{h}}(\mathfrak{p}))=0$, in particular,
$B(\mathfrak{p},[C_{\mathfrak{h}}(\mathfrak{p}),C_{\mathfrak{h}}(\mathfrak{p})])=0$. If $C_{\mathfrak{h}}(\mathfrak{p})$ is semisimple this implies
$B(\mathfrak{p},C_{\mathfrak{h}}(\mathfrak{p}))=0$ and the condition  2) implies $B(C_{\mathfrak{h}}(\mathfrak{p}),[\mathfrak{p},\mathfrak{p}]_{\mathfrak{h}})=0$.

If $\dim(\mathfrak{p})=1$, then $[\mathfrak{p},\mathfrak{p}]_{\mathfrak{h}}=0$ and all is clear. If  $\dim(\mathfrak{p})\geq 2$,
and $\mathfrak{p}$ is $\ad(\mathfrak{h})$-irreducible then
$[\mathfrak{h},\mathfrak{p}]=\mathfrak{p}$, hence $B(\mathfrak{p},C_{\mathfrak{h}}(\mathfrak{p}))=0$ and, consequently, the condition~3) implies
$B([\mathfrak{p},\mathfrak{p}]_{\mathfrak{h}},C_{\mathfrak{h}}(\mathfrak{p}))=0$.
If $\mathfrak{p}$ has no 1-dimensional $\ad(\mathfrak{h})$-invariant submodule, then we also get
$[\mathfrak{h},\mathfrak{p}]=\mathfrak{p}$, hence, the condition 4) implies
$B([\mathfrak{p},\mathfrak{p}]_{\mathfrak{h}},C_{\mathfrak{h}}(\mathfrak{p}))=0$.
\end{proof}

\begin{remark}\label{rem3}
If $B(\mathfrak{h},\mathfrak{m})=0$, then we get $B([\mathfrak{p},\mathfrak{p}]_{\mathfrak{h}},C_{\mathfrak{h}}(\mathfrak{p}))=0$ for any
$\ad(\mathfrak{h})$-invariant submodule $\mathfrak{p}$ such that $[\mathfrak{p},\mathfrak{p}]\subset \mathfrak{p}\oplus \mathfrak{h}$.
Moreover, if $B$ is non-degenerate on $\mathfrak{p}$, then
$\mathfrak{h}=[\mathfrak{p},\mathfrak{p}]_{\mathfrak{h}}\oplus C_{\mathfrak{h}}(\mathfrak{p})$ in this case.
Indeed, if $X\in \mathfrak{h}$ satisfies $B([\mathfrak{p},\mathfrak{p}]_{\mathfrak{h}},X)=0$, then
$0=B([\mathfrak{p},\mathfrak{p}]_{\mathfrak{h}},X)=B([\mathfrak{p},\mathfrak{p}],X)=
B(\mathfrak{p},[\mathfrak{p},X])$, hence $X\in C_{\mathfrak{h}}(\mathfrak{p})$.
\end{remark}

Now, we apply Proposition \ref{totgeodet} to get one well known result (see e.~g. Theorem 1.14 in~\cite{Gor96}).

\begin{corollary}\label{totgeodetn}
Let $(M=G/H, \rho)$ be a geodesic orbit Riemannian space and let $K$ be a connected closed subgroup of $G$, normalized by $H$.
Then the orbit $K/K\cap H$ of $K$ trough the point $o=eH$ is a totally geodesic submanifold of $M$, hence, a geodesic orbit Riemannian manifold.
In particular, the orbits of the radical $R(G)$, of the nilradical $N(G)$, and of a Levi group $S$, that are stable under $H$, through the point $o=eH$ are
geodesic orbit Riemannian manifolds.
\end{corollary}

\begin{proof}
Since $K$ is normalized by $H$, then its Lie algebra $\mathfrak{k}$ is $\Ad(H)$-invariant.
The Lie algebra $\mathfrak{h}\cap \mathfrak{k}$ (being $\Ad(H)$-invariant) admits an $\Ad(H)$-invariant complement $\mathfrak{p}$ in $\mathfrak{k}$.
On the other hand, the $\Ad(H)$-invariant Lie subalgebra $\mathfrak{h}+\mathfrak{k}$ admits an $\Ad(H)$-invariant complement $\mathfrak{q}$ in $\mathfrak{g}$.
Hence, we can take $\mathfrak{m}:=\mathfrak{p} \oplus \mathfrak{q}$ as an $\Ad(H)$-invariant complement to $\mathfrak{h}$ in $\mathfrak{g}$.
It is easy to see that $\mathfrak{k}=(\mathfrak{k} \cap \mathfrak{h})\oplus(\mathfrak{k} \cap \mathfrak{m})=(\mathfrak{k} \cap \mathfrak{h})\oplus \mathfrak{p}$,
and $[\mathfrak{p},\mathfrak{p}]\subset \mathfrak{p} \oplus (\mathfrak{k} \cap \mathfrak{h})\subset \mathfrak{p} \oplus  \mathfrak{h}$.
Hence, it suffices to apply Proposition \ref{totgeodet}.
\end{proof}
\smallskip

In general, we do not know the details of the reductive complement $\mathfrak{m}$ in the proof of this corollary
(since $\mathfrak{m}$ depends on the subgroup $K$). Hence, we could not obtain general results on the isometry group $G'$ with respect to which
a given orbit of $K$ is geodesic orbit space.
But it is possible to get such results (using Proposition \ref{totgeodet}) for some concrete subgroups $K$.
\medskip

It should be noted also that the space $K/K\cap H$ in Corollary \ref{totgeodetn} {\it should not be a geodesic orbit Riemannian space}
(in this case the full isometry group is more extensive than $K$, since $K/K\cap H$ is geodesic orbit Riemannian manifold).

\begin{example}\label{examplenil}
Let us consider the case $K=N(G)$. We know that $\mathfrak{n(g)} \cap \mathfrak{h}=0$, hence $\mathfrak{p}=\mathfrak{n(g)}$. Clear that
$[\mathfrak{p},\mathfrak{p}]_{\mathfrak{h}}=0$. Hence, $\mathfrak{h}'=0$, $\mathfrak{k}=\mathfrak{p}=\mathfrak{n(g)}$ in terms of Proposition \ref{totgeodet}.
If we suppose that $C_{\mathfrak{h}}(\mathfrak{p})=\mathfrak{h}$ in order to have the equality  $\mathfrak{h}'+C_{\mathfrak{h}}(\mathfrak{p})=\mathfrak{h}$,
then $N(G)$ with the induces Riemannian metric should be a geodesic orbit space. But it is possible if and only if all operators
$\ad(Y)|_{\mathfrak{n(g)}}$, $Y\in \mathfrak{n(g)}$, are skew-symmetric. Indeed, since the isotropy algebra assumed to be trivial, then the GO-condition
(see Lemma \ref{GO-criterion}) is the following: for any $X,Y \in \mathfrak{n(g)}$ the equality $([X,Y],X)=0$ holds and
$\ad(Y)|_{\mathfrak{n(g)}}$ is a skew-symmetric operator.
On the other hand, all operators $\ad(Y)|_{\mathfrak{n(g)}}$, $Y\in \mathfrak{n(g)}$, are nilpotent. Therefore, $\mathfrak{n(g)}$ should be commutative.
Hence, for any non-commutative $\mathfrak{p}=\mathfrak{n(g)}\subset \mathfrak{m}$ we get $\mathfrak{h}'+C_{\mathfrak{h}}(\mathfrak{p})\neq \mathfrak{h}$.
\end{example}

\begin{example}\label{examplesem}
Let us consider the weakly symmetric space $G/H=Sp(n+1)U(1)/Sp(n)U(1)$ (diffeomorphic to the sphere $S^{4n+3}$).
This space admits a 3-parameter family of
invariant Riemannian metrics that are geodesic orbit. The subalgebra $sp(n+1)$ is stable under $\Ad_G(H)$,
hence $Sp(n+1)/Sp(n)$ with induced metrics is totally geodesic (in fact it coincides with the original manifolds $S^{4n+3}$) and hence geodesic
orbit Riemannian manifolds. On the other hand, the space of $Sp(n+1)$-invariant geodesic orbit metrics on $Sp(n+1)/Sp(n)$ is only 2-dimensional.
Therefore, for a suitable choice of $Sp(n+1)U(1)$-invariant metrics (in fact, for almost all such metrics) on $S^{4n+3}$,
the orbit of $Sp(n+1)$ through the point $o=eSp(n)U(1)$
is not a geodesic orbit space.
One can find details in~\cite{Nik2013}, see also Remark 2 in \cite{DuKoNi}.
\end{example}
\medskip

We emphasize in particular the fact that
the study of general geodesic orbit Riemannian manifolds
does not reduced completely to the study of nilmanifolds and homomogeneous spaces $G/H$ with semisimple $G$.

\section{On the structure of the radical and the nilradical}

Let us consider an arbitrary geodesic orbit Riemannian space $(G/H, \rho)$ and an $\Ad(H)$-invariant decomposition
$\mathfrak{g}=\mathfrak{h}\oplus \mathfrak{m}$ with $B(\mathfrak{h},\mathfrak{m})=0$, where $B$ is the Killing form of~$\mathfrak{g}$, and $\rho$
generated by an inner product $(\cdot,\cdot)$ on  $\mathfrak{m}$.
\smallskip

Let $C_{\mathfrak{g}}(\mathfrak{h})$ be the centralizer of $\mathfrak{h}$ in $\mathfrak{g}$.
It is clear that $C_{\mathfrak{g}}(\mathfrak{h})=C_{\mathfrak{g}}(\mathfrak{h})\cap \mathfrak{h}+C_{\mathfrak{g}}(\mathfrak{h}) \cap \mathfrak{m}$,
where $C_{\mathfrak{g}}(\mathfrak{h})\cap \mathfrak{h}=C_{\mathfrak{h}}(\mathfrak{h})$ is the center of $\mathfrak{h}$.
Obviously, $C_{\mathfrak{g}}(\mathfrak{h})$  and
$C_{\mathfrak{g}}(\mathfrak{h})\oplus [\mathfrak{h},\mathfrak{h}]$ are subalgebra in the Lie algebra $\mathfrak{g}$ with
$[C_{\mathfrak{g}}(\mathfrak{h}), [\mathfrak{h},\mathfrak{h}]]$=0.

\begin{lemma}\label{comni3}
For any $Y\in C_{\mathfrak{g}}(\mathfrak{h})$ the operator $\ad(Y)|_{\mathfrak{m}}$ is skew-symmetric.
If $X\in C_{\mathfrak{g}}(\mathfrak{h}) \cap \mathfrak{m}$, then
$([X,Z]_{\mathfrak{m}},X)=0$ for any $Z\in \mathfrak{g}$.
\end{lemma}

\begin{proof}
Let us prove the first assertion. Note that $[Y,\mathfrak{m}]\subset \mathfrak{m}$ for any $Y\in C_{\mathfrak{g}}(\mathfrak{h})$,
that follows from $B([Y,\mathfrak{m}],\mathfrak{h})=-B(\mathfrak{m},[Y,\mathfrak{h}])=-B(\mathfrak{m},0)=0$.
Obviously,  $\ad(Y)|_{\mathfrak{m}}$ is skew-symmetric for any $Y\in \mathfrak{h}$.
Now, take any $Y\in  C_{\mathfrak{g}}(\mathfrak{h})\cap \mathfrak{m}$.
For any $X\in \mathfrak{m}$ there is
$Z\in \mathfrak{h}$ such that $0=([X+Z,Y]_{\mathfrak{m}},X)=([X,Y]_{\mathfrak{m}},X)+([Z,Y],X)=([X,Y]_{\mathfrak{m}},X)$, because $[\mathfrak{h},Y]=0$.
Therefore, $\ad(Y)|_{\mathfrak{m}}$ is skew-symmetric in this case too.

Let us prove the second assertion. For any $X\in C_{\mathfrak{g}}(\mathfrak{h}) \cap \mathfrak{m}$ there is
$Z\in \mathfrak{h}$ such that $0=([X+Z,Y]_{\mathfrak{m}},X)=([X,Y]_{\mathfrak{m}},X)+([Z,Y],X)=
([X,Y]_{\mathfrak{m}},X)-(Y,[Z,X])=([X,Y]_{\mathfrak{m}},X)$ for any
$Y\in\mathfrak{m}$, since $[\mathfrak{h},X]=0$. For $Y\in \mathfrak{h}$, the equality $([X,Y]_{\mathfrak{m}},X)=0$ is obvious.
\end{proof}

\begin{corollary}\label{cor1}
The Lie algebra $\mathfrak{k}:=C_{\mathfrak{g}}(\mathfrak{h})\oplus [\mathfrak{h},\mathfrak{h}]$ is compactly embedded in $\mathfrak{g}$.
The Killing form $B$ of $\mathfrak{g}$ is non-positive on $\mathfrak{k}$. Moreover,
$B(Y,Y)=0$ for $Y\in \mathfrak{k}$ if and only if  $Y$ is in the center of~$\mathfrak{g}$.
\end{corollary}

\begin{proof}
We know that the operator $\ad(Y)|_{\mathfrak{m}}$ is skew-symmetric for all
$Y\in C_{\mathfrak{g}}(\mathfrak{h})$.
By definition, we get also
$[Y,\mathfrak{h}]=0$. If we extend $(\cdot,\cdot)$ to the $\Ad(H)$-invariant product oh $\mathfrak{g}$ with $(\mathfrak{h},\mathfrak{m})=0$,
then $\ad(Y)$ is skew-symmetric. The same we can say about any $Y\in \mathfrak{h}$, hence for all $Y \in \mathfrak{k}$.
This means that $\mathfrak{k}$ is compactly embedded in $\mathfrak{g}$. It is obviously also,
that $B(Y,Y)=\trace(\ad(Y)\ad(Y))=\sum_i ([Y,[Y, E_i]],E_i)=-\sum_i ([Y, E_i],[Y,E_i])\leq 0$, where $\{E_i\}$ is any $(\cdot,\cdot)$-orthonormal base
in $\mathfrak{g}$, with
$B(Y,Y)=0$ if and only if $\ad(Y)=0$.
\end{proof}

\begin{prop}\label{nilpcompli1} Under the above assumptions we get the following:

{\rm 1)} Any $\ad(\mathfrak{h})$-invariant complement $\mathfrak{p}$ to $[\mathfrak{g},\mathfrak{r(g)}]$ in
$\mathfrak{r(g)}$ is  in $C_{\mathfrak{g}}(\mathfrak{h})$;

{\rm 2)} $C_{\mathfrak{g}}(\mathfrak{h})\cap  \mathfrak{n(g)}$ is the center of $\mathfrak{g}$;

{\rm 3)} Any $\ad(\mathfrak{h})$-invariant complement $\mathfrak{q}$ to $[\mathfrak{g},\mathfrak{r(g)}]$ in
$\mathfrak{n(g)}$ is in the center of $\mathfrak{g}$.
\end{prop}

\begin{proof} It is clear that $[\mathfrak{h}, \mathfrak{p}]\subset [\mathfrak{g},\mathfrak{r(g)}]$.
On the other hand, $[\mathfrak{h}, \mathfrak{p}]\subset \mathfrak{p}$,
that proves 1).

Next, take any $Y\in C_{\mathfrak{g}}(\mathfrak{h})\cap  \mathfrak{n(g)}$. By Lemma \ref{comni3} we get that
the operator $\ad(Y)|_{\mathfrak{m}}$ is skew-symmetric, whereas $[Y,\mathfrak{h}]=0$.
On the other hand, $Y\in \mathfrak{n(g)}$ and $\ad(Y)$ should be nilpotent.
Hence, $Y$ is in the center of $\mathfrak{g}$. It is clear also that any central elements of $\mathfrak{g}$
is situated both in $\mathfrak{n(g)}$ and in $C_{\mathfrak{g}}(\mathfrak{h})$, hence, we get 2).

Finally, from $[\mathfrak{h}, \mathfrak{q}]\subset \mathfrak{q}\cap [\mathfrak{g},\mathfrak{r(g)}]$ we get $\mathfrak{q}\subset C_{\mathfrak{g}}(\mathfrak{h})$.
Hence, 2) implies 3).
\end{proof}
\smallskip

\begin{corollary}\label{comni4}
Let $\mathfrak{q}$ be the $(\cdot,\cdot)$-orthogonal complement to $[\mathfrak{r(g)},\mathfrak{g}]\,\bigl(\subset \mathfrak{n(g)}\bigr)$ in
$\mathfrak{r(g)}\cap \mathfrak{m}$, then
$\ad(Y)|_{\mathfrak{m}}$ is skew-symmetric for any $Y\in \mathfrak{q}$, $\mathfrak{q}$ commutes both with $\mathfrak{h}$ and with the orthogonal complement to
$[\mathfrak{r(g)},\mathfrak{g}]$ in $\mathfrak{m}$; in particular, $[\mathfrak{q},\mathfrak{q}]=0$.
\end{corollary}

\begin{proof}
Indeed, $\mathfrak{q} \subset C_{\mathfrak{g}}(\mathfrak{h})$ by Proposition \ref{nilpcompli1},
$\ad(Y)|_{\mathfrak{m}}$ is skew-symmetric for any $Y\in \mathfrak{q}$
by Lemma~\ref{comni3}.
Since $[\mathfrak{q},\mathfrak{g}]\subset [\mathfrak{r(g)},\mathfrak{g}]$ and
$\ad(Y)|_{\mathfrak{m}}$ is skew-symmetric for any $Y\in \mathfrak{q}$, we see that such $Y$ commutes with any $X\in \mathfrak{m}$ orthogonal to
$[\mathfrak{r(g)},\mathfrak{g}]$.
\end{proof}

\begin{remark}\label{comni5}
Vectors from $\mathfrak{q}$ could act non-trivially only on $[\mathfrak{r(g)},\mathfrak{g}]\subset \mathfrak{n(g)}$,
but in general $\mathfrak{q}$ is non-trivial subspace, see Example \ref{ex1} below.
\end{remark}

\begin{theorem}\label{strucrad1}
Let $(G/H, \rho)$ be a geodesic orbit space. Then {\rm(}in the above notation{\rm)} we have the following assertions.

{\rm 1)} The Killing form $B$ is non-positive on any $\ad(\mathfrak{h})$-invariant complement $\mathfrak{p}$ to $[\mathfrak{g},\mathfrak{r(g)}]$ in
$\mathfrak{r(g)}$, and $B(X,X)=0$, $X\in \mathfrak{p}$, if and only if $X$ is central in $\mathfrak{g}$.

{\rm 2)} The Killing form $B$ is negatively definite on any $\ad(\mathfrak{h})$-invariant complement $\mathfrak{q}$ to $\mathfrak{n(g)}$ in
$\mathfrak{r(g)}$.

{\rm 3)} $\mathfrak{n(g)}=\{X\in \mathfrak{g}\,|\,[B(X,\mathfrak{g})=0\}$ and we have the following
 $(\cdot,\cdot)$-orthogonal sum: $\mathfrak{n(g)}=[\mathfrak{g},\mathfrak{r(g)}]\oplus \mathfrak{l}$, where
$\mathfrak{l}$ is a central subalgebra of $\mathfrak{g}$.
\end{theorem}

\begin{proof}
By 1) of Proposition \ref{nilpcompli1}, any $\ad(\mathfrak{h})$-invariant complement $\mathfrak{p}$ to $[\mathfrak{g},\mathfrak{r(g)}]$ in
$\mathfrak{r(g)}$ is  in $C_{\mathfrak{g}}(\mathfrak{h})$.
By Corollary~\ref{cor1}, $B(X,X)\leq 0$  for $X\in \mathfrak{p}$, whereas $B(X,X)=0$ if and only if $X$ is in the center of
$\mathfrak{g}$. This proves 1).

Since the center is  a subset of the nilradical, $B$ is negatively defined
on any $\ad(\mathfrak{h})$-invariant complement $\mathfrak{q}$ to $\mathfrak{n(g)}$ in
$\mathfrak{r(g)}$, that proves 2).

From this we get also
$\mathfrak{n(g)}=\{X\in \mathfrak{g}\,|\,[B(X,\mathfrak{g})=0\}$, since
$\mathfrak{n(g)}\subset \{X\in \mathfrak{g}\,|\,[B(X,\mathfrak{g})=0\}\subset \mathfrak{r(g)}$.
By 3) of Proposition \ref{nilpcompli1} we get that
$\mathfrak{l}$ is  a central subalgebra of $\mathfrak{g}$.
\end{proof}
\smallskip

It is a good place to recall the structure of the nilradical in $\mathfrak{g}$ for a geodesic orbit space $(G/H, \rho)$.
The following result was obtained  by C.~Gordon
\cite[Theorem 2.2]{Gor96} for nilmanifold and by J.A.~Wolf \cite[Proposition 13.1.9]{W1} in general case.
For the reader's convenience, we add a short proof
of this remarkable result.

\begin{prop}[C.~Gordon~---~J.A.~Wolf]\label{strucnilr}
Let $(G/H, \rho)$ be a geodesic orbit space. Then the nilradical $\mathfrak{n(g)}$ of the Lie algebra $\mathfrak{g}=\operatorname{Lie}(G)$
is commutative or two-step nilpotent.
\end{prop}

\begin{proof}
Suppose that $\mathfrak{n(g)}$ is not commutative, i.~e. $[\mathfrak{n(g)}, \mathfrak{n(g)}]\neq 0$.
Let $\mathfrak{a}$ be the orthogonal complement to $[\mathfrak{n(g)}, \mathfrak{n(g)}]$ in $\mathfrak{n(g)}\subset \mathfrak{m}.$
It is easy to see that for any $X\in \mathfrak{a}$ the operator $\ad(X)|_{[\mathfrak{n(g)}, \mathfrak{n(g)}]}$ is skew-symmetric.
Indeed, by Lemma \ref{GO-criterion}, there is
$Z\in \mathfrak{h}$ such that $0=([X+Z,Y],X)=([X,Y],X)+([Z,Y],X)=([X,Y],X)$ for any
$Y\in  [\mathfrak{n(g)}, \mathfrak{n(g)}]$, since $[Z,Y]\in [\mathfrak{n(g)}, \mathfrak{n(g)}] \perp \mathfrak{a}$.
On  the other hand, $\ad(X)|_{[\mathfrak{n(g)}, \mathfrak{n(g)}]}$ is nilpotent by the Engel theorem. Therefore,
$[X,Y]=0$ for any $X\in \mathfrak{a}$ and any $Y\in  [\mathfrak{n(g)}, \mathfrak{n(g)}]$.
It is easy to see that $[\mathfrak{a},\mathfrak{a}],
[\mathfrak{a},[\mathfrak{a},\mathfrak{a}]],\dots, [\mathfrak{a},[\dots ,[\mathfrak{a},[\mathfrak{a},\mathfrak{a}]]\dots]], \dots$
also commute with
$[\mathfrak{n(g)}, \mathfrak{n(g)}]$.  But the subspace $\mathfrak{a}$ generates $\mathfrak{n(g)}$, since $\mathfrak{n(g)}$ is nilpotent
(see e.~g. \cite{Bourb}, exercise~4 to Section~4 of Chapter~I).
Hence, $[\mathfrak{n(g)},[\mathfrak{n(g)}, \mathfrak{n(g)}]]=0$ and $\mathfrak{n(g)}$ is two-step nilpotent.
\end{proof}

\begin{remark} From Theorem \ref{strucrad1} and Proposition \ref{strucnilr} we see that
$[\mathfrak{g},\mathfrak{r(g)}]$ is either commutative or two-step nilpotent Lie algebra for any geodesic orbit space.
\end{remark}
\medskip

Theorem 1.15 in  \cite{Gor96} claims that $N(G)\cdot S$ acts transitively on a geodesic orbit Riemannian space $(G/H, \rho)$,
where $N(G)$ is the largest connected nilpotent normal subgroup and $S$ is any semisimple Levi factor in $G$,
that equivalent to $\mathfrak{n(g)}+\mathfrak{s}+\mathfrak{h}=\mathfrak{g}$.
Examples \ref{ex1} and~\ref{ex2} below refute this assertion
(nevertheless, it is valid for the case of naturally reductive metrics, see Theorem 3.1 in~\cite{Gor85}).
In particular, we get $\mathfrak{n(g)}\neq \mathfrak{r(g)}\cap \mathfrak{m}$ in these examples.
Proposition \ref{gonil2} shows the way to construct examples of such kind.

\begin{example}\label{ex1} Let us consider the standard action of the Lie algebra $u(n)$ on $\mathbb{R}^{2n}$ (by skew-symmetric matrices), $n \geq 2$.
We will use the notation $A(X)$ for the action of  $A\in u(n)$ on $X\in \mathbb{R}^{2n}$.
Note that for any $X\in \mathbb{R}^{2n}$ there is a non-zero $A\in u(n)$ such that $A(X)=0$ and $A$ is not from  $su(n)$
(because the standard action of $U(n)$ on $S^{2n-1}$ has the isotropy group $U(n-1)\not\subset SU(n)$).

Now, consider $\mathfrak{r}=\mathbb{R}^{2n+1}=\mathbb{R}^{2n} \rtimes \mathfrak{z}$, a semidirect sum
of Lie algebras, where $\mathfrak{z}$ acts on $\mathbb{R}^{2n}$ as the center
of $u(n)$ on $\mathbb{R}^{2n}$. This Lie algebra is solvable (but it is non nilpotent!) with the abelian nilradical $\mathbb{R}^{2n}$.
Moreover, we have also a natural action of $su(n)$ on $\mathfrak{r}=\mathbb{R}^{2n+1}$:
$[Z,X+Y]=Z(X)$ for every $Z\in su(n)$, $X\in \mathbb{R}^{2n}$, $Y\in \mathfrak{z}$.
Hence, we get semidirect sum  $\mathfrak{g}=\mathfrak{r} \rtimes su(n)$ of Lie algebras.
Note also that $\mathfrak{r}$ could be defined as the radical of the Lie algebra
$\mathbb{R}^{2n} \rtimes u(n)=\mathfrak{r} \rtimes su(n)$.

Supply $\mathfrak{r}=\mathbb{R}^{2n+1}$ with the standard Euclidean inner product.
Let us prove that $G/H$ with the corresponding invariant Riemannian metrics, where $H=SU(n)$, is geodesic orbit Riemannian space.

It suffices to prove that for any $X\in  \mathbb{R}^{2n}$ and for any $Y\in \mathfrak{z}=\mathbb{R}$ there is $Z\in su(n)$ such that
$([X+Y+Z, X_1+Y_1],X+Y)=([X,Y_1]+[Y+Z,X_1],X)=0$ for every $X_1\in  \mathbb{R}^{2n}$ and $Y_1\in \mathfrak{z}=\mathbb{R}$.
Since $su(n)$ and $\mathfrak{z}=\mathbb{R}$ act on $\mathbb{R}^{2n}$ by skew-symmetric operators, this equivalent to $[Y+Z,X]=0$.

If $Y=0$, then we can take $Z=0$. Now, suppose that $Y\neq 0$.
We know that  there a non-zero $A\in u(n)$ such that $A(X)=0$ and $A\not\in su(n)$ for the standard action of $u(n)$ on $\mathbb{R}^{2n}$.
Let $A=Y_2+Z$, where $Y_2 \in  \mathfrak{z}=\mathbb{R}$ and $Z\in su(n)$.
Since $Y_2\neq 0$ we may assume (without loss of generality) that $Y_2=Y$ by multiplying $A$ by a suitable constant. Hence, we have found a desirable $Z\in su(n)$.

Hence, we get an example of  geodesic orbit Riemannian  space $(G/H, \mu)$, where $\mathfrak{n(g)}\neq \mathfrak{r(g)}\cap \mathfrak{m}$.
\end{example}

\begin{remark}
Note that the constructed GO-space is simply Euclidean space $\mathbb{R}^{2n+1}$. See also the case~4) of Theorem 4.4 in \cite{KV}.
\end{remark}
\medskip

Now we consider another construction of geodesic orbit spaces.
Let $\mathfrak{n}$ be a two-step nilpotent Lie algebra with an inner product $(\cdot,\cdot)$.
Denote by $\mathfrak{z}$ and $\mathfrak{a}$ the Lie subalgebra $[\mathfrak{n},\mathfrak{n}]$ and the $(\cdot,\cdot)$-orthogonal complement to it in $\mathfrak{n}$.
It is clear that $\mathfrak{z}$ is central in $\mathfrak{n}$. For any $Z\in \mathfrak{z}$ we consider the operator $J_Z:\mathfrak{a} \rightarrow \mathfrak{a}$
defined by the formula $(J_Z(X),Y)=([X,Y],Z)$.
It is clear that $J:\mathfrak{z} \rightarrow so(\mathfrak{a})$ is an injective map.

The Lie algebra of the isometry group of the corresponding Lie group $N$ with the corresponding left-invariant Riemannian metric $\mu$ is
a semidirect sum of $\mathfrak{n}$ with $D(\mathfrak{n})$, the algebra of skew-symmetric derivations of $\mathfrak{n}$, see e.~g. Theorem 4.2 in \cite{Wolf1962}
or \cite{Wil}. Recall the following important result.

\begin{prop}[C.~Gordon \cite{Gor96}]\label{gonil1}
$(N,\mu)$ is geodesic orbit Riemannian manifold if and only if for any $X\in \mathfrak{z}$ and $Y\in \mathfrak{a}$ there is $D\in D(\mathfrak{n})$ such that
$[D,X]=D(X)=0$, $[D,Y]=D(Y)=J_X(Y)$.
\end{prop}

Now, suppose that $D(\mathfrak{n})=\mathfrak{c}\oplus \mathfrak{d}$, where $\mathfrak{c}$ is one-dimensional central ideal in $D(\mathfrak{n})$.
Hence we can define a solvable Lie algebra $\mathfrak{sol}=\mathfrak{n}\oplus \mathfrak{c}$, that is semidirect sum ($[U,X]=U(X)$ for $X\in \mathfrak{n}$ and
$U\in \mathfrak{c}$). We can extend $(\cdot,\cdot)$ to $\mathfrak{sol}$ assuming $(\mathfrak{n},\mathfrak{c})=0$ and choosing any inner product on $\mathfrak{c}$,
Since $[\mathfrak{c},\mathfrak{d}]=0$, then $ \mathfrak{d}$ is an algebra of skew-symmetric derivations of $\mathfrak{sol}$.

\begin{prop}\label{gonil2}
In the above notation and assumptions, suppose that

{\rm 1)} for every $X\in \mathfrak{z}$ and every $Y\in \mathfrak{a}$, there is $D_1\in \mathfrak{d}\subset D(\mathfrak{n})$ such that
$[D_1,X]=D_1(X)=0$, $[D_1,Y]=D_1(Y)=J_X(Y)$;

{\rm 2)} the set $\{D\in D(\mathfrak{n})\,|\,D(X)=D(Y)=0\}$ does not lie in~$\mathfrak{d}$ for every $X\in \mathfrak{z}$ and every $Y\in \mathfrak{a}$.

Then the Lie group $Sol$ with a left-invariant Riemannian metric, corresponding to the metric Lie algebra $(\mathfrak{sol},(\cdot,\cdot))$
is a geodesic orbit Riemannian manifold.
\end{prop}

\begin{proof} Since
$\mathfrak{d}$ is an algebra of skew-symmetric derivations of $\mathfrak{sol}$, it suffices to prove the following:
for any $X\in \mathfrak{z}$, $Y\in \mathfrak{a}$, and $Z \in \mathfrak{c}$ there is $D\in \mathfrak{d}$ such that
$$
([X+Y+Z+D,U],X+Y+Z)=0 \quad \mbox{ for any } \quad U\in \mathfrak{sol}.
$$
Since $X$ is central in $\mathfrak{n}$, $[\mathfrak{c},\mathfrak{c}]=0$,
$[\mathfrak{c},\mathfrak{d}]=0$, and $[\mathfrak{sol},\mathfrak{sol}]\subset \mathfrak{n}$, it is
equivalent to the following:
$$
0=([Y+Z+D,U],X+Y)=([Y+Z+D,U_1],X+Y)+([X+Y,U_2],X+Y),
$$
where $U_1\in \mathfrak{n}$, $U_2\in \mathfrak{c}$, $U=U_1+U_2$.
Note that $([X+Y,U_2],X+Y)=0$ because $U_2\in  \mathfrak{c}\subset D(\mathfrak{n})$ is skew-symmetric on $\mathfrak{n}$.
On the other hand, $([Y+Z+D,U_1],X+Y)=([Y,U_1],X)+([Z+D,U_1],X+Y)=(J_X(Y), U_1)-([Z+D,X+Y],U_1)$
by the definition of $J_X$, and we get the equivalent equation
$0=(J_X(Y), U_1)-([Z+D,X+Y],U_1)$. Since $U$ is arbitrary, it suffices to prove that
for any $X\in \mathfrak{z}$, $Y\in \mathfrak{a}$, and $Z \in \mathfrak{c}$ there is $D\in \mathfrak{d}$ such that $[Z+D,X+Y]=J_X(Y)$.

From assumption 1) we know that there is $D_1\in \mathfrak{d}$ such that
$[D_1,X]=D_1(X)=0$ and $[D_1,Y]=D_1(Y)=J_X(Y)$. Let consider the set $C(X+Y):=\{D\in D(\mathfrak{n})\,|\,D(X)=D(Y)=0\}$.
This set is not trivial and is not in $\mathfrak{d}$ by assumption 2).

Now it is clear that for any $W\in \mathfrak{c}$ there is $V\in \mathfrak{d}$ such that $V+W \in C(X+Y)$.
Indeed, such $V$ does exists at least for one nontrivial $W$ by assumption 2). Now, it suffices to use multiplication by a constant, because
$\mathfrak{c}$ is 1-dimensional.

For a given $Z \in \mathfrak{c}$ we can choose
$V\in \mathfrak{d}$ such that $Z+V \in C(X+Y)$. Now, let consider $D=V+D_1$. Since $D_1\in \mathfrak{d}$, then $D \in \mathfrak{d}$  and
$[Z+D,X+Y]=J_X(Y)$.
\end{proof}

\begin{example}\label{ex2}
Consider the 13-dimensional metric Lie algebra of Heisenberg type $\mathfrak{n}$, see details in \cite{DuKo2003}.
In this case $\dim (\mathfrak{z})=5$, $\dim (\mathfrak{a})=8$, $D(\mathfrak{n})=so(5)\oplus \mathbb{R}$.
It it known that for any $X\in \mathfrak{z}$ and $Y\in \mathfrak{a}$ there is (a unique) $D_1\in so(5)$ such that
$[D_1,X]=D_1(X)=0$ and $[D_1,Y]=D_1(Y)=J_X(Y)$, see p.~92 in \cite{DuKo2003}.
In fact this is proved earlier in \cite{Riehm}. In \cite{DuKo2003}, a suitable $D_1$ is determined by a (unique) solution $d$ of
a linear system $Bd=b$, where $B$ is $(10\times 10)$-matrix with elements that depend linearly on $X$ and $Y$.
It is easy to see also that
the set $\{D\in D(\mathfrak{n})\,|\,D(X)=D(Y)=0\}$ is not situated in $so(5)$ for any $X\in \mathfrak{z}$ and any $Y\in \mathfrak{a}$, because
any nontrivial $D\in so(5)$, such that $D(X)=D(Y)=0$, is determined by a non-trivial solution $d$ of the homogeneous linear system $Bd=0$,
where $B$ is non-degenerate, that is impossible,
see details in \cite{DuKo2003}.
Therefore, according to Proposition \ref{gonil2}, we get a 14-dimensional solvmanifold that is a geodesic orbit Riemannian space.
\end{example}
\medskip

In the last part of this section we will show that every geodesic orbit Riemannian space $(G/H, \rho)$
naturally generates another geodesic orbit Riemannian space $(\widetilde{G}/\widetilde{K},\widetilde{\rho})$, that
has a remarkable property: the group
$N\bigl(\widetilde{G}\bigr)\cdot \widetilde{S}$ acts transitively on $(\widetilde{G}/\widetilde{K},\widetilde{\rho})$,
where $N\bigl(\widetilde{G}\bigr)$ is the the largest connected nilpotent normal subgroup and $\widetilde{S}$ is any semisimple Levi factor in $\widetilde{G}$.

\begin{prop}\label{nilpcompli1n} For any geodesic orbit space $(G/H, \rho)$ we have the decomposition
$\mathfrak{m}=(C_{\mathfrak{g}}(\mathfrak{h}) \cap \mathfrak{m}) \oplus [\mathfrak{h}, \mathfrak{m}]$,
which is both $B$-orthogonal and $(\cdot,\cdot)$-orthogonal. Moreover, $C_{\mathfrak{g}}(\mathfrak{h})$ is $B$-orthogonal to
$[\mathfrak{h}, \mathfrak{g}]=[\mathfrak{h}, \mathfrak{h}]\oplus[\mathfrak{h}, \mathfrak{m}]$.
\end{prop}

\begin{proof}
Since $\mathfrak{m}$ is an $\ad(\mathfrak{h})$-module, then we get the decomposition into the sum of linear spaces
$\mathfrak{m}=(C_{\mathfrak{g}}(\mathfrak{h}) \cap \mathfrak{m}) \oplus [\mathfrak{h}, \mathfrak{m}]$, which is standard (see e.~g.
Lemma 14.3.2 in \cite{HilNeeb}). Since the inner product $(\cdot,\cdot)$ is $\ad(\mathfrak{h})$-invariant, then
$0=-([\mathfrak{h},C_{\mathfrak{g}}(\mathfrak{h}) \cap \mathfrak{m}], \mathfrak{m})=(C_{\mathfrak{g}}(\mathfrak{h}) \cap \mathfrak{m},[\mathfrak{h},\mathfrak{m}])$.
The same is valid for the Killing form $B$:
$0=-B([\mathfrak{h},C_{\mathfrak{g}}(\mathfrak{h}) \cap \mathfrak{m}], \mathfrak{m})=
B(C_{\mathfrak{g}}(\mathfrak{h}) \cap \mathfrak{m},[\mathfrak{h},\mathfrak{m}])$.
Moreover, we have $0=-B([\mathfrak{h},C_{\mathfrak{g}}(\mathfrak{h})],\mathfrak{g})=B(C_{\mathfrak{g}}(\mathfrak{h}),[\mathfrak{h}, \mathfrak{g}])=
B(C_{\mathfrak{g}}(\mathfrak{h}),[\mathfrak{h}, \mathfrak{h}]\oplus[\mathfrak{h}, \mathfrak{m}])$.
\end{proof}
\medskip

\begin{remark}
If a Levi subalgebra $\mathfrak{s}$ is $\ad(\mathfrak{h})$-invariant, then
$C_{\mathfrak{g}}(\mathfrak{h}) \cap \mathfrak{m}=(C_{\mathfrak{g}}(\mathfrak{h}) \cap \mathfrak{r(g)})\oplus (C_{\mathfrak{g}}(\mathfrak{h}) \cap \mathfrak{s})$
as linear space. Note also that $B(\mathfrak{r(g)},[\mathfrak{g}, \mathfrak{g}])=B(\mathfrak{r(g)},[\mathfrak{g}, \mathfrak{r(g)}]\oplus \mathfrak{s})=0$ and
$B(\mathfrak{n(g)},\mathfrak{g})=0$.
\end{remark}
\smallskip

In what follows we will need

\begin{prop}\label{quathom1}
Let $G$ be a  Lie group and $K$ be a closed connected subgroup of $G$ such that its Lie algebra $\mathfrak{k}=\operatorname{Lie}(K)$ is compactly embedded in
$\mathfrak{g}=\operatorname{Lie}(G)$. Then the homogeneous space $G/K$ admits
a $G$-invariant Riemannian metric $\mu$. If $L$ is a maximal normal subgroup of $G$ in $K$, then $G/L$ acts transitively and effectively by isometries
on $(G/K,\mu)$, and the group $K/L$ is compact if and only if $G/L$ is a closed subgroup in the full connected isometry group of $(G/K=(G/L)/(K/L),\mu)$.
\end{prop}

\begin{proof}
By the assumption and the definition of compactly embedded subalgebra, there is an inner product $(\cdot,\cdot)$ on the Lie algebra $\mathfrak{g}$
relative to which the operators $\ad(X):\mathfrak{g} \mapsto \mathfrak{g}$, $X\in \mathfrak{k}$, are skew-symmetric.
Equivalently, the closure of $\Ad_G(K)$ in $\Aut(\mathfrak{g})$ is a compact.
Since $\Aut(\mathfrak{g})$ carries the relative topology of $GL(\mathfrak{g})$, then the closure of $\Ad_G(K)$ in $GL(\mathfrak{g})$ is a compact.
By Theorem 3.16 in \cite{ChEb}, the space $G/K$ admits a $G$-invariant Riemannian metric $\mu$, that is a  $(G/L)$-invariant Riemannian metric
on the effective space $(G/L)/(K/L)$. For the last assertion see e.~g. Theorem 1.1 in \cite{DMM}.
\end{proof}
\smallskip

Now we point out a closed connected subgroup $K$ in $G$ with the Lie algebra $C_{\mathfrak{g}}(\mathfrak{h})\oplus [\mathfrak{h},\mathfrak{h}]$.

\begin{prop}\label{quathom2}
Let us consider the normalizer $N_G(\mathfrak{h})=\{a\in G\,|\, \Ad(a)(\mathfrak{h})\subset \mathfrak{h}\}$ of $\mathfrak{h}$ in the Lie group $G$.
Then it is a closed subgroup in $G$ with the Lie algebra $C_{\mathfrak{g}}(\mathfrak{h})\oplus [\mathfrak{h},\mathfrak{h}]$.
The same property has its unit component $N_G(\mathfrak{h})_0$.
\end{prop}

\begin{proof} The closeness of $N_G(\mathfrak{h})$ follows directly from its definition.
Hence, $N_G(\mathfrak{h})$ is a Lie subgroup of $G$.
Let us show that  $N_G(\mathfrak{h})$ corresponds to the Lie subalgebra $C_{\mathfrak{g}}(\mathfrak{h})\oplus [\mathfrak{h},\mathfrak{h}]$ in
$\mathfrak{g}$. Indeed, the Lie algebra of $N_G(\mathfrak{h})$ could be characterized as
$\operatorname{Lie}(N_G(\mathfrak{h}))=\{X\in \mathfrak{g}\,|\, [X,\mathfrak{h}]\subset \mathfrak{h}\}$.
It is clear that $\mathfrak{h} \subset \operatorname{Lie}(N_G(\mathfrak{h}))$. Further, consider any $X\in \operatorname{Lie}(N_G(\mathfrak{h}))$.
There are $U\in \mathfrak{h}$ and $V\in \mathfrak{m}$ such that $X=U+V$. Since $U\in \operatorname{Lie}(N_G(\mathfrak{h}))$, then
$V\in \operatorname{Lie}(N_G(\mathfrak{h}))$. Recall that $[\mathfrak{h},\mathfrak{m}]\subset \mathfrak{m}$. Hence,
$V\in \operatorname{Lie}(N_G(\mathfrak{h}))$ is equivalent to $[V,\mathfrak{h}]=0$, i.~e. to $V\in C_{\mathfrak{g}}(\mathfrak{h})$.
Obviously also, that $V\in C_{\mathfrak{g}}(\mathfrak{h})$ implies $X=U+V\in \operatorname{Lie}(N_G(\mathfrak{h}))$. It is also obvious that
$\operatorname{Lie}(N_G(\mathfrak{h}))=C_{\mathfrak{g}}(\mathfrak{h})\oplus [\mathfrak{h},\mathfrak{h}]$.
\end{proof}
\smallskip

From Lemma \ref{comni3} and Proposition \ref{quathom2}, we easily get

\begin{corollary}\label{symmmet}
The inner product $(\cdot,\cdot)$, generating
the metric of a geodesic orbit Riemannian space $(G/H,\rho)$, is not only $\Ad(H)$-invariant but also $\Ad(N_G(H_0))$-invariant, where
$N_G(H_0)$ is the normalizer of the unit component $H_0$ of the group $H$ in $G$.
\end{corollary}

This property is well known for weakly symmetric spaces (see e.~g. Lemma 2 in~\cite{Yakimova})
and for generalized normal homogeneous spaces (Corollary 6 in~\cite{BerNik}).
\smallskip

Let $K$ be the unit component of $N_G(\mathfrak{h})$
with the Lie algebra $\mathfrak{k}=C_{\mathfrak{g}}(\mathfrak{h})\oplus [\mathfrak{h},\mathfrak{h}]$.
Let $L$ be the maximal normal subgroup of $G$ in $K$. By Corollary \ref{cor1}, $\mathfrak{k}$ is compactly embedded in~$\mathfrak{g}$.
By Proposition \ref{quathom1}, the space $G/K$ admits a $G$-invariant Riemannian metric
and every such metric is $\widetilde{G}$-invariant on the space $\widetilde{G}/\widetilde{K}$, where
$\widetilde{K}=K/L$ and $\widetilde{G}=G/L$.

Clear that $\mathfrak{k}=\mathfrak{h}\oplus (\mathfrak{k}\cap \mathfrak{m})$.
Put $\widetilde{\mathfrak{m}}:=[\mathfrak{h}, \mathfrak{m}]$,   both $B$-orthogonal and  $(\cdot,\cdot)$-orthogonal complement
to $C_{\mathfrak{g}}(\mathfrak{h})\cap \mathfrak{m}=\mathfrak{k}\cap \mathfrak{m}$ in $\mathfrak{m}$, see Proposition \ref{nilpcompli1n}.
Obviously, $\widetilde{\mathfrak{m}}$ is $\Ad_G(K)$-invariant.

By Lemma \ref{comni3}, for any
$Y\in \mathfrak{k}$ the operator $\ad(Y)|_{\mathfrak{m}}$ is skew-symmetric, the same we can say about $\ad(Y)|_{\widetilde{\mathfrak{m}}}$.
Therefore, $(\cdot, \cdot)_1$, the restriction of $(\cdot, \cdot)$ to $\widetilde{\mathfrak{m}}$,
correctly determines some $G$-invariant Riemannian metric $\widetilde{\rho}$
on the homogeneous space $G/K=\widetilde{G}/\widetilde{K}$.
Moreover, there is a natural Riemannian submersion (generated by an orthogonal projection $\mathfrak{m} \rightarrow \widetilde{\mathfrak{m}}$)

\begin{equation}\label{rsubm1}
(G/H, \rho) \rightarrow (\widetilde{G}/\widetilde{K},\widetilde{\rho}).
\end{equation}

\smallskip

\begin{theorem}\label{rsubm2} The homogeneous Riemannian space $(G/K=\widetilde{G}/\widetilde{K},\widetilde{\rho})$ is geodesic orbits.
Moreover, all $\Ad(\widetilde{K})$-invariant submodules
in $\widetilde{\mathfrak{m}}$ have dimension $\geq 2$, $\widetilde{\mathfrak{n(g)}}=[\widetilde{\mathfrak{g}},\widetilde{\mathfrak{r(g)}}]$, and
$\widetilde{\mathfrak{r(g)}}=\widetilde{\mathfrak{n(g)}}\oplus \bigl(\widetilde{\mathfrak{r(g)}}\cap \widetilde{\mathfrak{k}}\,\bigr)$,
where $\widetilde{\mathfrak{g}}=\operatorname{Lie}(\widetilde{G})$, $\widetilde{\mathfrak{r(g)}}$ and $\widetilde{\mathfrak{n(g)}}$ are
the radical and the nilradical of $\widetilde{\mathfrak{g}}$ respectively.
Consequently, the group
$N\bigl(\widetilde{G}\bigr)\cdot \widetilde{S}$ acts transitively on $(\widetilde{G}/\widetilde{K},\widetilde{\rho})$,
where $N\bigl(\widetilde{G}\bigr)$ is the largest connected nilpotent normal subgroup and $\widetilde{S}$ is any semisimple Levi factor in $\widetilde{G}$.
\end{theorem}

\begin{proof} Consider $X\in \widetilde{\mathfrak{m}}$. Since $(G/H, \rho)$ is geodesic orbit, there is $Z\in \mathfrak{h}$ such that
$([X+Z,Y]_{\mathfrak{m}},X) =0$ for all $Y\in \mathfrak{m}$, see Lemma \ref{GO-criterion}.
If we consider only $Y\in \widetilde{\mathfrak{m}}$, then we can rewrite this equality as follows:
$([X+Z,Y]_{\widetilde{\mathfrak{m}}},X)_1 =0$. Since $\mathfrak{h} \subset \mathfrak{k}$,
then $(G/K=\widetilde{G}/\widetilde{K},\widetilde{\rho})$ is geodesic orbits according to Lemma \ref{GO-criterion}.
In fact, any geodesic in $(G/K=\widetilde{G}/\widetilde{K},\widetilde{\rho})$ is a projection of some geodesic in $(G/H, \rho)$
through submersion (\ref{rsubm1}).

It is easy to see that
$C_{\mathfrak{g}}(\mathfrak{k})\subset {\mathfrak{k}}$, because even $C_{\mathfrak{g}}(\mathfrak{h})\subset {\mathfrak{k}}$.
In particular, it implies that there is no 1-dimensional $\ad(\mathfrak{k})$-irreducible subspace in $\widetilde{\mathfrak{m}}$, since every such subspace
is a subset of $C_{\mathfrak{g}}(\mathfrak{k})$.  Applying Proposition \ref{nilpcompli1} to the space $(G/K=\widetilde{G}/\widetilde{K},\widetilde{\rho})$,
we get $\widetilde{\mathfrak{n(g)}}=[\widetilde{\mathfrak{g}},\widetilde{\mathfrak{r(g)}}]$ and
$\widetilde{\mathfrak{r(g)}}=\widetilde{\mathfrak{n(g)}}\oplus \bigl(\widetilde{\mathfrak{r(g)}}\cap \widetilde{\mathfrak{k}}\,\bigr)$.
\end{proof}

\medskip

\medskip

\section{Relations to representations with non-trivial principal \\isotropy algebras}

As in the previous section, we consider an arbitrary geodesic orbit Riemannian space $(G/H, \rho)$ and an $\Ad(H)$-invariant decomposition
$\mathfrak{g}=\mathfrak{h}\oplus \mathfrak{m}$, where $\mathfrak{m}:=\{X\in \mathfrak{g}\,|\,B(X,\mathfrak{h})=0\}$,
$B$ is the Killing form of~$\mathfrak{g}$, and $\rho$ generated by an inner product $(\cdot,\cdot)$ on  $\mathfrak{m}$.
Let us consider  the operator
$A:\mathfrak{m} \rightarrow \mathfrak{m}$, related with the Killing form by $B(X,Y)=(AX,Y)$, see Lemma \ref{h.2} and Remark \ref{eigkill1}.
All results in this section are obtained under these assumptions.
\smallskip

Recall that the operator $A$ is invertible on the subspace $\mathfrak{m}_{\pm}:=\oplus_{\alpha \neq 0} A_{\alpha}$,
where $A_{\alpha}$ is the eigenspace of the operator $A$ with the eigenvalue $\alpha$. In what follows we assume that
$A^{-1}X \in \mathfrak{m}_{\pm}$ for all $X\in \mathfrak{m}_{\pm}$.

\begin{prop}\label{gokilf}
The following assertions hold:

{\rm 1)} For $X\in \mathfrak{m}_{\pm}$, the GO-condition from Lemma \ref{GO-criterion} {\rm(}there is $Z \in \mathfrak{h}$ such that
$([X+Z,Y]_{\mathfrak{m}},X) =0$ for all $Y\in \mathfrak{m}${\rm)}  is equivalent to the following one:
there is $Z \in \mathfrak{h}$ such that $[X+Z,A^{-1}X]\subset A_0\oplus \mathfrak{h}$.

{\rm 2)} For any $X\in A_{\alpha}$ and $Y\in A_{\beta}$, where $0\neq \alpha\neq \beta\neq 0$, there is $Z \in \mathfrak{h}$ such that
$(\beta-\alpha)[X,Y]=\beta[Z,X]+\alpha[Z,Y]$. In particular, $[A_{\alpha},A_{\beta}]\subset A_{\alpha}\oplus A_{\beta}$.

{\rm 3)} Let $X,Y \in A_{\alpha}$, $\alpha\neq 0$, are such that $([\mathfrak{h},X],Y)=0$, then $[X,Y]\in A_0\oplus A_{\alpha}$.
\end{prop}

\begin{proof} Let us prove 1). It is easy to see that
$0=([X+Z,Y]_{\mathfrak{m}},X) =B([X+Z,Y], A^{-1}X)=-B(Y,[X+Z,A^{-1}X]$.
Since $Y \in \mathfrak{m}$ is arbitrary and $B$ is non-degenerate on $\mathfrak{m}_{\pm}$, we get the required assertion.

Let us prove 2).  By 1), for the vector $X+Y$ there is
$Z \in \mathfrak{h}$ such that $[X+Y+Z,A^{-1}(X+Y)]\subset A_0\oplus \mathfrak{h}$. Clear that
$$
[X+Y+Z,A^{-1}(X+Y)]=[X+Y+Z, \alpha^{-1} X+\beta^{-1}Y]=
$$
$$
\frac{1}{\alpha\beta}((\alpha-\beta)[X,Y]+\beta[Z,X]+\alpha[Z,Y]).
$$
Since
$A_{\alpha}$ and $A_{\beta}$ are $\ad(\mathfrak{h})$-invariant  and $[A_{\alpha},A_{\beta}]\subset \mathfrak{m}$ by Lemma \ref{h.2}, we get
$(\alpha-\beta)[X,Y]+\beta[Z,X]+\alpha[Z,Y]\in A_0$. But $A_0$ is in ideal in $\mathfrak{g}$ and the restriction of $\ad(X)$  to
$\oplus_{\gamma \neq \alpha} A_{\gamma}$ is skew-symmetric by  Lemma \ref{skew-symmetryLemma},
hence, $\ad(X)(A_{\beta})\subset \mathfrak{m}_{\pm}=\oplus_{\gamma \neq 0} A_{\gamma}$. Consequently, $(\alpha-\beta)[X,Y]+\beta[Z,X]+\alpha[Z,Y]=0$, that
proves the assertion.

Let us prove 3). Take any $U\in A_{\beta}$, $\beta \not\in \{0,\alpha\}$. Then, using 2), we get for some $Z\in \mathfrak{h}$
$$
([X,Y],U)=\beta^{-1}B([X,Y],U)=-\beta^{-1}B(Y,[X,U])=
$$
$$
\beta^{-1}(\alpha-\beta)^{-1}B(Y,\beta[Z,X]+\alpha[Z,U])=\alpha(\alpha-\beta)^{-1}(Y,[Z,X])=0,
$$
since $[Z,U] \in A_{\beta}$ and $([\mathfrak{h},X],Y)=0$.
Note also that $([\mathfrak{h},X],Y)=0$ implies $0=B([\mathfrak{h},X],Y)=B(\mathfrak{h},[X,Y])$. Therefore, we see that $[X,Y] \in A_0\oplus A_{\alpha}$,
that proves the third assertion.
\end{proof}
\smallskip

\begin{remark}
Note that $A_0=\mathfrak{n(g)}$ by Theorem \ref{strucrad1}. If we have $\mathfrak{m}=A_{\alpha}$ for some $\alpha < 0$, then the space $(G/H, \rho)$
is normal homogeneous, because $(\cdot, \cdot)=\alpha^{-1} B$ on $\mathfrak{m}$ and $\rho$ is generated by the invariant  positive definite
form $\alpha^{-1} B$ on the Lie algebra $\mathfrak{g}$. Of course, such space is geodesic orbit: to check it, it suffices to put $Z=0$ for all $X \in \mathfrak{m}$
in the GO-condition. On the other hand,
if  $\mathfrak{m}=A_{\alpha}$ for some $\alpha > 0$, then the space $(G/H, \rho)$ is symmetric of non-compact type (see e.~g. \cite{Hel}),
hence, naturally reductive. Indeed, $\mathfrak{r(g)}$ is trivial by Theorem \ref{strucrad1}
(indeed, $\mathfrak{n(g)}=A_0$ is trivial, hence $\mathfrak{r(g)}$ is also trivial) and $\mathfrak{h}$ should coincide with maximal compactly embedded
subalgebra in semisimple Lie algebra $\mathfrak{g}$.
\end{remark}

Now, we consider the case when there are several eigenspaces of the operator $A$ with non-zero eigenvalues.
For any $V \in \mathfrak{g}$ we denote the centralizer of $V$ in $\mathfrak{h}$ by $C_{\mathfrak{h}}(V)$.

\begin{theorem}\label{centaction1}
For any $X\in A_{\alpha}$ and $Y\in A_{\beta}$, where $0\neq \alpha\neq \beta\neq 0$, there are $Z_1 \in C_{\mathfrak{h}}(X)$  and $Z_2 \in C_{\mathfrak{h}}(Y)$
such that
$[X,Y]=[Z_2,X]+[Z_1,Y]$. In particular, if $C_{\mathfrak{h}}(Y)=\{0\}$ {\rm(}respectively,  $C_{\mathfrak{h}}(X)=\{0\}${\rm)}, then $[X,Y] \in A_{\beta}$
{\rm(}respectively, $[X,Y] \in A_{\alpha}${\rm)}. Consequently, the equality $C_{\mathfrak{h}}(Y)=C_{\mathfrak{h}}(X)=\{0\}$ implies $[X,Y]=0$.
\end{theorem}

\begin{proof}
By 2) of Proposition \ref{gokilf}, for given $X\in A_{\alpha}$ and $Y\in A_{\beta}$, there is $Z \in \mathfrak{h}$ such that
$(\beta-\alpha)[X,Y]=\beta[Z,X]+\alpha[Z,Y]$. Similarly, for $X\in A_{\alpha}$ and $-Y\in A_{\beta}$, there is $Z' \in \mathfrak{h}$ such that
$(\beta-\alpha)[X,-Y]=\beta[Z',X]+\alpha[Z',-Y]$. From the above two equalities we get $\beta[Z+Z',X]+\alpha[Z-Z',Y]=0$.
It is clear that $Z+Z' \in C_{\mathfrak{h}}(X)$ and $Z-Z'\in C_{\mathfrak{h}}(Y)$. Now, consider
$Z_1:=\frac{\alpha}{2(\beta-\alpha)}(Z+Z')$ and $Z_2:=\frac{\beta}{2(\beta-\alpha)}(Z-Z')$,
then $Z=\frac{\beta-\alpha}{\alpha} Z_1+\frac{\beta-\alpha}{\beta} Z_2$ and
$(\beta-\alpha)[X,Y]=\beta[Z,X]+\alpha[Z,Y]$ implies $[X,Y]=[Z_2,X]+[Z_1,Y]$.
\end{proof}
\medskip

This theorem allows us to apply new tools to the study of geodesic orbit Riemannian manifolds.
These tools related to the problem of classifying principal orbit types for linear actions of
compact Lie groups. If a compact linear Lie group $K$ acts on some finite-dimensional vector space $V$
(in other terms, we have a representation of a Lie group on the space~$V$),
then almost all points of $V$ are situated on the orbits of $K$, that are pairwise isomorphic as $K$-manifolds.
Such orbits are called orbits in general position. The isotropy groups  of all points on such orbits are conjugate
in $K$, the class of these isotropy groups is called a {\it principal isotropy group} for the linear group $K$
and the corresponding Lie algebra is called a {\it principal isotropy algebra} or a {\it stationary subalgebra of points in general position}.
Roughly speaking, the principal isotropy algebra is trivial for general linear Lie groups $K$, but it is not the case for some special linear groups.
\smallskip

The classification of linear actions of simple compact connected Lie groups with
non-trivial connected principal isotropy subgroups, has already
been carried out in \cite{Kram}, \cite{Hsiang}, and  \cite{Ela1}.
See details in \S 3 of Chapter 5 in \cite{Hsiangbook}.
In \S 4 and \S 5 of Chapter 5 in \cite{Hsiangbook}, one can find
a description of more general compact connected Lie groups with
non-trivial connected principal isotropy subgroups (see e.~g. Theorem (V.7') in \cite{Hsiangbook}).

In 1972, A.G.~\`{E}lashvili described all stationary subalgebras of points in general position for the case in which $G$ is
simple and $V$ is arbitrary (see \cite{Ela1}) and also if $G$ is semisimple and $V$ is irreducible (see \cite{Ela2}).
According to \cite{Il}, a linear action of a semisimple group $K$ on
a finite-dimensional vector space $V$ is said to be locally strongly effective if
every simple normal subgroup of $K$ acts nontrivially on every irreducible submodule
$W \subset V$.
In \cite{Il}, D.G.~Il'inskii obtained a complete description of the locally strongly effective actions
of a semisimple complex algebraic group $K$ on a complex vector space $V$ with nontrivial stationary subalgebra of points in general position.
\smallskip

\begin{theorem}\label{centaction2}
Let $\chi: H \to O(\mathfrak{m})$ be the isotropy representation for $(G/H, \rho)$ and
let $\chi_{\mathfrak{p}}$  be its irreducible subrepresentation on a submodule $\mathfrak{p} \subset A_{\alpha}$, $\alpha \neq 0$.
Then, the following assertions hold:

{\rm 1)} For any $A_{\beta}$, where $\alpha\neq \beta\neq 0$,  we get $[A_{\beta}, \mathfrak{p}]\subset A_{\beta}\oplus \mathfrak{p}$;

{\rm 2)} If $\mathfrak{p}^{\perp}$ is the $B$-orthogonal {\rm(}or, equivalently, $(\cdot,\cdot)$-orthogonal{\rm)} complement to
$\mathfrak{p}$ in $A_{\alpha}$, then
$[\mathfrak{p}, \mathfrak{p}^{\perp}]\subset A_0\oplus A_{\alpha}$;

{\rm 3)} If the principal isotropy algebra of the representation $\chi_{\mathfrak{p}}$ is trivial, then $[A_{\beta}, \mathfrak{p}]\subset  \mathfrak{p}$;
\end{theorem}

\begin{proof} By Theorem \ref{centaction1}, for any $X\in \mathfrak{p} \subset A_{\alpha}$ and $Y\in A_{\beta}$
there are $Z_1 \in C_{\mathfrak{h}}(X)$  and $Z_2 \in C_{\mathfrak{h}}(Y)$
such that $[X,Y]=[Z_2,X]+[Z_1,Y]$. Therefore, $[X,Y] \in A_{\beta}\oplus \mathfrak{p}$.

The inclusion $[\mathfrak{p}, \mathfrak{p}^{\perp}]\subset A_0\oplus A_{\alpha}$ follows from 3) of Proposition \ref{gokilf},
since $ \mathfrak{p}$ is $\ad(\mathfrak{h})$-invariant.

If the principal isotropy algebra of the representation $\chi_{\mathfrak{p}}$ is trivial, then for almost all $X\in \mathfrak{p}$, the linear space
$C_{\mathfrak{h}}(X)$ is trivial, hence $Z_1=0$ and $[X,Y] \in \mathfrak{p}$. By continuity, $[X,Y] \in \mathfrak{p}$ holds for all
$X\in \mathfrak{p}$.
\end{proof}

\begin{corollary}\label{centaction3}
If, in the above notations, $[A_{\beta}, \mathfrak{p}]\not\subset  \mathfrak{p}$, then the principal isotropy algebra of the representation
$\chi_{\mathfrak{p}}:H \rightarrow O(\mathfrak{p})$
is non-trivial, hence, has dimensions $\geq 1$.
\end{corollary}
\smallskip

Corollary \ref{centaction3} shows that the algebraic structure of geodesic orbit Riemannian spaces is very special.
This observation could help to understand  geodesic orbit Riemannian spaces on a more deep level.

\section{Conclusion}

Finally, we propose some open questions, that (at least in our opinion) are quite interesting and important
for the theory of geodesic orbit Riemannian spaces.

\begin{quest}\label{que0}
Produce new examples of geodesic orbit Riemannian solvmanifolds with using of Proposition \ref{gonil2}.
\end{quest}

Compare with Example \ref{ex2}.

\begin{quest}\label{que1}
Classify geodesic orbit Riemannian spaces $(G/H,\rho)$ with small number of eigenvalues of the operator $A$ {\rm(}see the previous section{\rm)}.
\end{quest}

It is interesting even the case of two different eigenvalues (one of which could be zero) in this question.

\begin{quest}\label{que2}
Classify geodesic orbit Riemannian spaces $(G/H,\rho)$ with small number {\rm(}$2,3,\dots${\rm)} of irreducible component in the isotropy representation.
\end{quest}

It should be noted that all isotropy irreducible spaces are naturally reductive and, hence, geodesic orbit.

\begin{quest}\label{que3}
Classify all homogeneous spaces $G/H$ such that all $G$-invariant Riemannian metrics on $G/H$ are geodesic orbit.
\end{quest}

Note that isotropy irreducible spaces give obvious examples of a required type. Commutative Lie groups constitute another obvious type of examples.
More interesting examples are weakly symmetric spaces, see details e.~g. in \cite{AV}.

\bigskip

{\bf Acknowledgements.}
The author is indebted to Prof.  Valerii Berestovskii, to Prof. Carolyn Gordon, and to Prof.  \`{E}rnest~Vinberg
for helpful discussions concerning this paper.

\vspace{10mm}

\bibliographystyle{amsunsrt}

\begin{thebibliography}{[99]}



\bibitem{AF}
I. Agricola,  A.C. Ferreira,
{\sl Tangent Lie groups are Riemannian naturally reductive spaces},
Adv. Appl. Clifford Algebras (2016), doi:10.1007/s00006-016-0660-3


\bibitem{AFF}
I. Agricola,  A.C. Ferreira, T. Friedrich,
{\sl The classification of naturally reductive homogeneous spaces in dimensions $n\leq 6$},
Differential Geom. Appl., 39  (2015), 52--92.


\bibitem{AA}
D.V.~Alekseevsky, A.~Arvanitoyeorgos,
{\sl Riemannian flag manifolds with homogeneous geodesics,} Trans. Amer. Math. Soc., 359 (2007), 3769--3789.

\bibitem{AN}
D.V. Alekseevsky, Yu.G. Nikonorov,  {\sl Compact Riemannian manifolds with homogeneous
geodesics}, SIGMA Symmetry Integrability Geom. Methods Appl., 5 (2009), 093, 16 pages.



\bibitem{AV}
D.N.~Akhiezer, \`{E}.B.~Vinberg,
{\sl Weakly symmetric spaces and spherical varieties,} Transform. Groups, 4 (1999), 3--24.


\bibitem{AVE}
E.M.~Andreev, \`{E}.B.~Vinberg, A.G.~\`{E}lashvili,
{\sl Orbits of greatest dimension in semi-simple linear Lie groups},
Funktsional. Anal. i Prilozen. 1(4) (1967), 3-–7 (1967) (Russian);  English translation: Funct. Anal. Appl. 1(4) (1967), 257-–261.


\bibitem{Barco}
V. del Barco,
{\sl Homogeneous geodesics in pseudo-Riemannian nilmanifolds,}
Advances in Geometry, 16(2) (2016), 175-–187.


\bibitem{BB1978}
L.~Berard Bergery, {\sl Sur la courbure des metriqes riemanniennes invariants
des groupes de Lie et des espaces homogenes},  Ann. Scient. Ec. Norm. Sup.,
11 (1978), 543--576.

\bibitem{BB1981}
L.~Berard Bergery,  {\sl Homogeneous Riemannian spaces of dimension $4$}. In:
G\`{e}om\`{e}trie Riemannienne en dimension $4$. S\`{e}minaire Arthur Besse.
Paris: Cedic, 1981.


\bibitem{Berg}
M.~Berger,
{\sl Les varietes riemanniennes homogenes normales a courbure strictement positive,}
Ann. Sc. Norm. Super. Pisa, Cl. Sci., IV Ser. 15(1961), 179--246.



\bibitem{BerNik}
V.N.~Berestovskii, Yu.G.~Nikonorov,
{\sl On $\delta$-homogeneous Riemannian manifolds,}
Differential Geom. Appl., 26(5) (2008), 514--535.


\bibitem{BerNik3}
V.N.~Berestovskii, Yu.G.~Nikonorov,
{\sl On $\delta$-homogeneous Riemannian manifolds, II,}
Sib. Math. J., 50(2) (2009), 214--222.

\bibitem{BerNikKF}
V.N.~Berestovskii, Yu.G.~Nikonorov,
{\sl Killing vector fields of constant length on Riemannian manifolds,}
Sib. Math. J.,
49 (3) (2008) 395--407.


\bibitem{BerNikClif}
V.N. Berestovskii, Yu.G. Nikonorov, {\sl Clifford-Wolf homogeneous Riemannian manifolds},
J. Differential Geom., 82(3), 2009, 467--500.

\bibitem{BerNik2012}
V.N.~Berestovskii, Yu.G.~Nikonorov,
{\sl Generalized normal homogeneous Riemannian metrics on spheres and projective spaces,}
Ann. Global Anal. Geom., 45(3) (2014), 167--196.


\bibitem{BKV}
J.~Berndt, O.~Kowalski, L.~Vanhecke,
{\sl Geodesics in weakly symmetric spaces,}
Ann. Global Anal. Geom. 15 (1997), 153--156.


\bibitem{Bes}
A.L.~Besse,
{\sl Einstein Manifolds}, Springer-Verlag, Berlin,
Heidelberg, New York, London, Paris, Tokyo, 1987.

\bibitem{Bourb}
N.~Bourbaki,
{\sl Lie groups and Lie algebras}.
Chapters 1--3. Translated from the French. Reprint of the 1989 English translation. Springer-Verlag, Berlin, 1998.

\bibitem{CHNT}
G.~Cairns, A.~Hini\'{c} Gali\'{c}, Y.~Nikolayevsky, I.~Tsartsaflis,
{\sl Geodesic bases for Lie algebras},
Linear Multilinear Algebra, 63 (2015), 1176--1194.

\bibitem{Ca}
\'E.~Cartan,
{\sl Sur une classe remarquable d'espaces de Riemann,} Bull. Soc. Math. de France, 54 (1926), 214--264;
55 (1927), 114--134.


\bibitem{ChEb}
J.~Cheeger, D.G.~Ebin,
{\sl Comparison theorems in Riemannian geometry},
AMS Chelsea Publishing, Providence, RI, 2008.


\bibitem{CNN2017}
Z.~Chen, Yu.~G.~Nikonorov, Yu.~V.~Nikonorova,
{\sl Invariant Einstein metrics on Ledger--Obata spaces}, Differential geometry and its applications, 50 (2017), 71--87,
see also arXiv:1605.04747.



\bibitem{ChenWolf2012}
Z. Chen,  J.A. Wolf,
{\sl Pseudo-Riemannian weakly symmetric manifolds,}
Ann. Glob. Anal. Geom., 41 (2012), 381-–390.



\bibitem{DMM}
I. Dotti Miatello, R.J. Miatello,
{\sl Transitive isometry groups with noncompact isotropy},
Pacific J. Math.
131(1) (1988), 167--178.

\bibitem{Du2010}
Z.~Du{\u s}ek,
{\sl The existence of homogeneous geodesics in homogeneous pseudo-Riemannian and affine manifolds},
J. Geom. Phys., 60 (2010), 687–-689.

\bibitem{DuKo2003}
Z.~Du{\u s}ek, O.~Kowalski,
{\sl Geodesic graphs on the 13-dimensional group of Heisenberg type},
Math. Nachr., 254--255 (2003), 87--96.

\bibitem{DuKoNi}
Z.~Du{\u s}ek, O.~Kowalski, S.~Nik{\u c}evi{\'c},
{\sl New examples of Riemannian g.o. manifolds in dimension 7},
Differential Geom. Appl., 21 (2004), 65--78.

\bibitem{Ela1}
A.G. \'{E}lashvili,
{\sl Canonical form and stationary subalgebras of points in general position for simple linear Lie groups},
Funktsional. Anal. i Prilozen. 6(1) (1972), 51-–62 (Russian); English translation: Funct. Anal. Appl. 6(1), (1972), 44--53.

\bibitem{Ela2}
A.G. \'{E}lashvili,
{\sl Stationary subalgebras of points of general position for irreducible linear Lie groups},
Funktsional. Anal. i Prilozen. 6(2) (1972), 65--78 (Russian);  English translation: Funct. Anal. Appl. 6(2), (1972) 139--148.



\bibitem{Jac}
N.~Jacobson,
{\sl Lie algebras,}
Interscience Publishers, New York-London, 1962.


\bibitem{Gor85}
C.~Gordon,
{\sl Naturally reductive homogeneous Riemannian manifolds,}
Canad. J. Math., 37(3) (1985), 467--487.

\bibitem{Gor96}
C.~Gordon,
{\sl Homogeneous Riemannian manifolds whose geodesics are orbits,} 155--174.
In: Progress
in Nonlinear Differential Equations. V.~20. Topics in geometry: in
memory of Joseph D'Atri. Birkh{\"a}user, 1996.


\bibitem{GorWil1985}
C.~Gordon, E.N.~Wilson
{\sl The fine structure of transitive Riemannian isometry groups. I,} Trans. Amer. Math. Soc., 289(1) (1985), 367--380.



\bibitem{Hel}
S.~Helgason,
{\sl Differential geometry, Lie groups, and symmetric spaces,}
Pure and Applied Math. 80,
Academic Press. New York, 1978, XV+628~pp.



\bibitem{HilNeeb}
J.~Hilgert, K.-H.~Neeb
{\sl Structure and geometry of Lie groups,}
Springer Monographs in Mathematics. Springer, New York, 2012, X+744~pp.


\bibitem{Hsiang}
W.C.~Hsiang, W.Y.~Hsiang,
{\sl Differentiable actions of compact connected classical group}
II, Ann. of Math., 92 (1970), 189--223,

\bibitem{Hsiangbook}
W.Y.~Hsiang,
{\sl Cohomology theory of topological transformation groups},
Springer Verlag, New York, Heidelberg, Berlin, 1975, X+164 pp.

\bibitem{Il}
D.G.~Il'inskii,
{\sl Stationary subalgebras in general position for locally strongly effective actions}, Matematicheskie Zametki, 88(5), (2010), 689--707 (Russian);
English translation: Mathematical Notes, 88(5), (2010), 661--677.


\bibitem{Kap}
A.~Kaplan
{\sl On the geometry of groups of  Heisenberg type},
Bull. London Math. Soc., 15 (1983), 35--42.


\bibitem{KN}
S.~Kobayashi, K.~Nomizu,
{\sl Foundations of differential
geometry,} Vol.~I -- A Wiley-Interscience Publication, New York,
1963; Vol.~II -- A Wiley-Interscience Publication, New York, 1969.

\bibitem{KV1}
O.~Kowalski, L.~Vanhecke,
{\sl Classification of five-dimensional naturally reductive spaces},
Math. Proc. Cambridge Phil. Soc., 97 (1985),  445--463.

\bibitem{KV}
O.~Kowalski, L.~Vanhecke,
{\sl Riemannian manifolds with homogeneous geodesics,}
Boll. Un. Mat. Ital. B (7),  5(1)  (1991),  189--246.


\bibitem{Kram}
M.~Kr\"{a}mer,
{\sl Hauptisotropiegruppen bei endlich dimensionalen Darstellungen kompakter halbeinfacher Liegruppen}, Diplomarbeit,
Bonn, 1966.



\bibitem{Mostov1956}
G.D.~Mostow,
{\sl Fully reducible subgroups of algebraic groups}, Amer. J. Math. 78 (1956), 200--221.




\bibitem{Nik2013n}
Yu.G.~Nikonorov,
{\sl Geodesic orbit manifolds and Killing fields of constant length,}
Hiroshima Math. J., 43(1) (2013), 129--137.

\bibitem{Nik2013}
Yu.G. Nikonorov, {\sl Geodesic orbit Riemannian metrics on spheres}, Vladikavkaz. Mat. Zh., 15(3) (2013), 67--76.

\bibitem{Nik2015}
Yu.G. Nikonorov, {\sl Killing vector fields of constant length on compact homogeneous Riemannian manifolds,}
Ann. Global Anal. Geom., 48(4) (2015), 305--330.



\bibitem{Riehm}
C. Riehm,
{\sl Explicit spin representations and Lie algebras of Heisenberg type}, J. London Math. Soc., 29 (1984), 46--62.


\bibitem{S}
A.~Selberg,
{\sl Harmonic Analysis and discontinuous groups in weakly symmetric Riemannian spaces,
with applications to Dirichlet series,} J. Indian Math. Soc. 20 (1956), 47--87.



\bibitem{Storm}
R.~Storm,
{\sl A new construction of naturally reductive spaces},
Prerpint, 2016, arXiv:1605.00432.

\bibitem{Taft}
E.J.~Taft,
{\sl Orthogonal conjugacies in associative and Lie algebras}, Trans. Amer. Math. Soc. 113 (1964), 18--29



\bibitem{Tam}
H.~Tamaru,
{\sl Riemannian geodesic orbit space metrics on fiber bundles,}
Algebras Groups Geom., 15 (1998) 55--67.



\bibitem{Ta}
H.~Tamaru,
{\sl Riemannian g.o. spaces fibered over irreducible symmetric spaces,}
Osaka J. Math., 36 (1999), 835--851.






\bibitem{Wil}
E.N.~Wilson, {\sl Isometry groups on homogeneous nilmanifolds},
Geom. Dedicata, 12(3) (1982), 337–-346.




\bibitem{Wolf1969}
J.A.~Wolf, {\sl A compatibility condition between invariant riemannian metrics and Levi--Whitehead decompositions on a coset space},
Trans. Amer. Math. Soc. 139 (1969), 429--442.

\bibitem{W1}
J.A.~Wolf,
{\sl Harmonic Analysis on Commutative Spaces,} American Mathematical Society, 2007.


\bibitem{Wolf1962}
J.A.~Wolf,
{\sl On locally symmetric spaces of non-negative curvature and certain other locally homogeneous spaces},
Commentarii Mathematici Helvetici,
37(1) (1962), 266--295.


\bibitem{Wust}
M.~W\"{u}stner,
{\sl Factoring a Lie group into a compactly embeddedand a solvable subgroup},
Monatsh. Math. 130(1) (2000),  49--55.

\bibitem{Yakimova}
O.S.~Yakimova.
{\sl Weakly symmetric Riemannian manifolds with a reductive isometry group},
(Russian)  Mat. Sb.  195(4)  (2004), 143--160;  English translation in Sb. Math.  195(3-4)  (2004),  599--614.



\bibitem{YanDeng}
Z. Yan, Sh. Deng,
{\sl Finsler spaces whose geodesics are orbits},
Differential Geom. Appl., 36  (2014), 1--23.



\bibitem{Zil96}
W.~Ziller,
{\sl Weakly symmetric spaces,} 355--368. In: Progress
in Nonlinear Differential Equations. V.~20. Topics in geometry: in
memory of Joseph D'Atri. Birkh{\"a}user, 1996.

\end{thebibliography}

\end{document}